\documentclass[twoside,10pt]{amsart}
  \usepackage{amsmath,amsfonts,amssymb,amsthm,mathtools,multirow,tabls,mathrsfs}
  \usepackage[a4paper,margin=1.1in]{geometry}  
  \usepackage[all]{xy}
  \usepackage[utf8]{inputenc}
  \usepackage[T1]{fontenc}
  \usepackage[shortlabels]{enumitem}
  \usepackage[colorlinks,citecolor=magenta,linkcolor=black]{hyperref}

  \numberwithin{equation}{section}

  \newtheorem{thm}{Theorem}[section]
  \newtheorem{lemma}[thm]{Lemma}
  \newtheorem{prop}[thm]{Proposition}
  \newtheorem{cor}[thm]{Corollary}
  \newtheorem{assertion}[thm]{Assertion}
  \newtheorem{thmi}{Theorem}     
  \newtheorem{conji}{Conjecture} 

  \theoremstyle{definition} 
  \newtheorem{defin}[thm]{Definition}
  \newtheorem{remark}[thm]{Remark}
  \newtheorem{example}[thm]{Example}

  \newcommand{\cE}{\mathscr E}
  \newcommand{\cI}{\mathscr I}
  \newcommand{\cL}{\mathscr L}
  \newcommand{\cO}{\mathscr O}
  \newcommand{\cX}{\mathscr X}

  \newcommand{\bb}[1]{\mathbf{#1}} 
  \newcommand{\FF}{\bb{F}}
  \newcommand{\PP}{\bb{P}}
  \newcommand{\QQ}{\bb{Q}}
  \newcommand{\ZZ}{\bb{Z}}

  \renewcommand{\phi}{\varphi}
  \newcommand{\isom}{\simeq} 
  \newcommand{\ra}{\longrightarrow}    
  \newcommand{\isomto}{\xrightarrow{\,\smash{\raisebox{-0.65ex}{\ensuremath{\scriptstyle\sim}}}\,}}

  \DeclareMathOperator{\Bl}{Bl}
  \DeclareMathOperator{\Def}{Def}
  \DeclareMathOperator{\Ext}{Ext}
  \DeclareMathOperator{\Gr}{Gr}
  \DeclareMathOperator{\Hom}{Hom}
  \DeclareMathOperator{\Pic}{Pic}
  \DeclareMathOperator{\Proj}{Proj}
  \DeclareMathOperator{\Sing}{Sing}
  \DeclareMathOperator{\Spec}{Spec}
  \DeclareMathOperator{\Sym}{Sym}

  \newcommand{\cExt}{{\mathscr E}\kern -.5pt xt}
  \newcommand{\cHom}{\mathscr{H}\kern -.5pt om}

  \newcommand{\wt}[1]{{\mathchoice%
    {\raisebox{1.5ex}{\resizebox{1.7ex}{!}{{}\hphantom{i}\ensuremath{{\sim}}}} \hspace{-1.7ex}{#1}}%
    {\smash{\raisebox{1.5ex}{\resizebox{1.7ex}{!}{{}\hphantom{i}\ensuremath{{\sim}}}}\hspace{-1.7ex}{#1 }}\vphantom{\tilde I}}%
    {\raisebox{1.1ex}{\resizebox{1.3ex}{!}{{}\hphantom{i}\ensuremath{{\sim}}}}\hspace{-1.3ex}{#1}}%
    {\raisebox{0.8ex}{\resizebox{1ex}{!}{{}\hphantom{i}\ensuremath{{\sim}}}}\hspace{-1ex}{#1}}%
  }}

  \newcommand{\swt}[1]{{\mathchoice%
    {\raisebox{0.9ex}{\resizebox{1.2ex}{!}{\ensuremath{{\sim}}}}\hspace{-1.4ex}{#1}}%
    {\smash{\raisebox{0.9ex}{\resizebox{1.2ex}{!}{\ensuremath{{\sim}}}}\hspace{-1.4ex}{#1 }}\vphantom{I}}%
    {\raisebox{0.7ex}{\resizebox{0.8ex}{!}{\ensuremath{{\sim}}}}\hspace{-0.9ex}{#1}}%
    {\raisebox{0.5ex}{\resizebox{1ex}{!}{{}\hphantom{i}\ensuremath{{\sim}}}}\hspace{-1ex}{#1}}%
  }}

  \newcommand{\wtcI}{\smash{\raisebox{1.5ex}{\hspace{0.7ex}\resizebox{1.2ex}{!}{\ensuremath{{\sim}}}}\hspace{-2.1ex}{\cI}}\vphantom{I}}

  \author[P.\ Achinger]{Piotr Achinger}
  \address{Institute of Mathematics, Polish Academy of Sciences, ul.\ Śniadeckich 8, 00-656 Warsaw, Poland}
  \email{pachinger@impan.pl}

  \author[J.\ Witaszek]{Jakub Witaszek}
  \address{Department of Mathematics,University of Michigan, Ann Arbor, MI 48109} 
  \email{jakubw@umich.edu}

  \author[M.\ Zdanowicz]{Maciej Zdanowicz}
  \address{\'Ecole Polytechnique F\'ed\'erale de Lausanne, MA C3 585, Station 8, CH-1015 Lausanne}
  \email{maciej.zdanowicz@epfl.ch}

  \title[Global Frobenius Liftability II]{Global Frobenius Liftability II: Surfaces and Fano threefolds}
  \date{\today}

\begin{document}

\begin{abstract}
  In this article, a sequel to \cite{PartI}, we continue the development of a comprehensive theory of Frobenius liftings modulo $p^2$. We study compatibility of divisors and closed subschemes with Frobenius liftings, Frobenius liftings of blow-ups, descent under quotients by some group actions, stability under base change, and the properties of associated $F$-splittings. Consequently, we characterise Frobenius liftable surfaces and Fano threefolds, confirming the conjecture stated in our previous paper.
\end{abstract}

\maketitle

\section{Introduction}
\label{s:intro}

Let $X$ be a smooth and proper algebraic variety over an algebraically closed field $k$ of characteristic $p>0$. We call $X$ \emph{$F$-liftable} if there exists a flat lifting $\wt X$ of $X$ over $W_2(k)$ together with a lifting $\wt F_X\colon \wt X\to \wt X$ of the absolute Frobenius $F_X\colon X\to X$; the pair $(\wt X, \wt F_X)$ is called a \emph{Frobenius lifting} of $X$. 
While \emph{local} liftings of Frobenius are related to the Cartier operator and are of utmost importance in characteristic $p$ geometry \cite{DeligneIllusie}, \emph{global} Frobenius liftings are very rare and heavily restrict the geometry of the variety $X$ in quite subtle ways (see \cite{BTLM}, where this was studied extensively in the case of homogeneous spaces, and \cite{xin16} for the case of surfaces). In the paper \cite{PartI}, we have begun a comprehensive study of $F$-liftable varieties, and proposed the following conjectural characterization. 

\begin{conji}[{\cite[Conjecture~1]{PartI}}] \label{conj:froblift}
  Let $X$ be a smooth projective variety over an algebraically closed field $k$ of characteristic $p>0$. If $X$ is $F$-liftable, then there exists a finite \'etale Galois cover $f\colon Y\to X$ such that the Albanese morphism $a_Y\colon  Y\ra {\rm Alb}(Y)$ is a toric fibration. In particular, if $X$ is simply connected, then it is a toric variety.
\end{conji}

\noindent
The goal of the present paper is to further the understanding of the geometric aspects of $F$-liftability. We provide several general results on Frobenius liftings, particularly in relationship with $F$-splittings, that often apply also to $X$ proper or singular. In the end, we put all of our findings to use and settle Conjecture~\ref{conj:froblift} in the following cases.

\begin{thmi} \label{thmi:main} 
  Conjecture~\ref{conj:froblift} is true if
  \begin{enumerate}[(1)]
    \item $\dim X\leq 2$ (see Theorem~\ref{thm:c1-surfaces}), or
    \item $X$ is a Fano threefold and $p\gg 0$, assuming a form of boundedness (see Theorem~\ref{thm:fano3} and Remark~\ref{rmk:fano3}(b)).
  \end{enumerate}
\end{thmi}

Verifying whether a given variety is $F$-liftable can be difficult in practice.  As shown in \cite{BTLM}, a smooth projective $F$-liftable variety satisfies \emph{Bott vanishing}
\begin{equation} \label{eqn:bott-van} 
  H^j(X, \Omega_X^i \otimes L) \quad (j>0, \, L \text{ ample}).
\end{equation}
While this property alone heavily restricts the geometry of an $F$-liftable variety, it is also cumbersome to verify, nor does it characterize $F$-liftable varieties \cite{totaro_bott_vanishing}. One of our goals is to provide a toolkit which enables one to deduce that a given variety is not $F$-liftable using geometric methods. 

From now on let us work over a fixed perfect field $k$ of characteristic $p>0$. Recall that an $F$-splitting on a $k$-scheme $X$ is an $\cO_X$-linear splitting of the map $\cO_X\to F_{X*} \cO_X$. There is a considerable amount of interplay between Frobenius liftings and $F$-splittings. First, if $X$ is normal, then every Frobenius lifting of $X$ induces an $F$-splitting. Second, properties of a Frobenius lifting are often reflected in the induced $F$-splitting (for example, being compatible with a divisor). However, Frobenius liftings tend to be more rigid than $F$-splittings; for example, every Frobenius lifting of a product of projective varieties arises as a product of Frobenius liftings of the factors (Corollary~\ref{cor:flift-products} and Proposition~\ref{prop:compatible_with_frobenius_on_products_preliminary}), while the corresponding fact is not true for $F$-splittings.

In \cite{zdanowicz}, the third author provided an explicit functorial construction of a lifting modulo $p^2$ associated to an $F$-splitting (see Theorem~\ref{thm:zdanowicz-canonical-lift}), which we exploit in \S\ref{s:flift-fsplit}. Using some properties of the Witt vector scheme $W_2(X)$, we show the following:

\begin{thmi}[{Theorem~\ref{thm:fsplit-liftings}}]
  Let $(\wt X, \wt F_X)$ be a Frobenius lifting of a $k$-scheme $X$, and let $\sigma$ be a Frobenius splitting on $X$. Let $\wt X(\sigma)$ be the canonical lifting of $X$ induced by $\sigma$ as in \cite{zdanowicz}. Then there is a canonical isomorphism $\wt X(\sigma)\isom \wt X$ of liftings of $X$. 
\end{thmi}

\noindent In particular, if $X$ is normal, then there exists at most one lifting $\wt X$ to which the Frobenius lifts (Corollary~\ref{cor:unique-f-lift}). 

As observed and applied in \cite{PartI}, one can descend the property of being $F$-liftable along fibrations $f\colon Y\to X$, or more generally certain maps for which $f^*\colon \cO_X\to Rf_*\cO_Y$ is a split monomorphism (\cite[Theorem~3.3.6]{PartI}, restated as Theorem~\ref{thm:descending-frob-lift}(ab) below). Using Frobenius splittings and canonical liftings described above, we can extend this to quotients by linearly reductive groups (Theorem~\ref{thm:descending-frob-lift}(c)). Similar ideas allow us to treat $F$-liftability of products (Corollary~\ref{cor:flift-products} and Proposition~\ref{prop:compatible_with_frobenius_on_products_preliminary}).

In a different direction, we deal with $F$-liftability of pairs (\S\ref{s:froblift-compatible}). We study Frobenius liftings of simple normal crossings pairs $(X, D)$ and introduce the related notion of a Frobenius lifting of a pair $(X, Z)$ where $Z\subseteq X$ is a closed subscheme (Definition~\ref{def:flift-compatible}). We show:

\begin{thmi}[{Proposition~\ref{prop:blow-ups}}] \label{thmi:flift-comp}
  Suppose that $X$ and $Z$ are smooth, and let $(\wt X, \wt Z)$ be a lifting of $(X, Z)$. Further, let $\swt\pi\colon \wt Y\to \wt X$ be the blow-up along $\wt Z$. Then a lifting of Frobenius $\wt F_X$ on $\wt X$ is compatible with $\wt Z$ if and only if it extends to a lifting of Frobenius on $\wt Y$. The induced lifting on $\wt F_Y$ is automatically compatible with the exceptional divisor $\swt\pi^{-1}(\wt Z)$.
\end{thmi}

The relationship between Frobenius liftings and $F$-splittings is particularly useful in a relative situation, which we explore in \S\ref{sec:f_liftings_and_splittings_relative}. In \S\ref{s:frob-base-change}, using Witt vectors of general rings we show a fundamental property of Frobenius liftings (Proposition~\ref{prop:frob-lift}): if $(\wt Y, \wt F_Y)$ is a Frobenius lifting of a scheme $Y$, $\phi\colon Z\to Y$ is a map, and $\wt Z$ is a lifting of $Z$ to $W_2(k)$, then $\phi\circ F_Y$ has a preferred lifting to a map $\psi\colon \wt Z\to \wt Y$. Using this, we prove a curious `base change' property of Frobenius liftings (see Corollary~\ref{cor:lifts-of-Cartesian-diagrams}), which in particular implies that the fibers of an ``$F$-liftable morphism'' are $F$-liftable. Further, we analyse the induced relative $F$-splittings. 

Let us end this introduction with a few words about the proof of Theorem~\ref{thmi:main}. The case of surfaces (\S\ref{s:surfaces}) is quickly reduced to rational and ruled surfaces, in which cases the question is still not completely trivial.  In the case of rational surfaces, our approach proceeds by induction with respect to the number of contractions necessary to obtain a Hirzebruch surface from a given surface. Note that for ruled surfaces our results do not entirely agree with \cite{xin16}, see Remark~\ref{rem:xin}. The case of Fano threefolds (\S\ref{s:fano3}) relies on the Mori–Mukai classification with about 100 families; see Remark~\ref{rmk:fano3} for a~discussion. Since many of these varieties naturally arise as blow-ups, Theorem~\ref{thmi:flift-comp} is particularly useful in treating their $F$-liftability.   The case which caused us the most trouble initially was the blow-up of $\PP^3_k$ along a twisted cubic, for which we used our study of relative $F$-splittings associated to $F$-liftable fibrations mentioned above.  Since the Mori--Mukai classification is not known in characteristic $p$, we can only apply it to a lift of a given $F$-liftable Fano threefold to characteristic zero. Descending the information obtained this way back to characteristic $p$ relies on a slightly delicate argument with Hilbert schemes (see Proposition~\ref{prop:fano-rigid-spec}).

Although a proof based on the classification might seem tedious and not very useful, we embarked on this task for two reasons. First, the proof confirms our belief that Frobenius liftability is very rare among smooth projective varieties, and provides substantial empirical evidence for the validity of Conjecture~\ref{conj:froblift}. But what is even more important is that the examples we have analyzed  were a source of inspiration in our study of general properties of $F$-liftable  varieties.   

\medskip
\noindent \emph{Throughout the paper, we work over a fixed perfect field $k$ of characteristic $p>0$.}

\subsection*{Acknowledgements}

We would like to thank Paolo Cascini, Nathan Ilten, Adrian Langer, Arthur Ogus, Vasudevan Srinivas, Nicholas Shepherd-Barron, and Jarosław Wiśniewski for helpful suggestions and comments. Part of this work was conducted during the miniPAGES Simons Semester at the Banach Center in Spring 2016.  The authors would like to thank the Banach Center for hospitality. P.A.\ was supported by NCN OPUS grant number UMO-2015/17/B/ST1/02634 and NCN SONATA grant number UMO-2017/26/D/ST1/00913. J.W.\ was supported by the Engineering and Physical Sciences Research Council [EP/L015234/1].  M.Z.\ was supported by NCN PRELUDIUM grant number UMO-2014/13/N/ST1/02673.  This work was partially supported by the grant 346300 for IMPAN from the Simons Foundation and the matching 2015--2019 Polish MNiSW fund. 


\section{\texorpdfstring{$F$-liftability and $F$-splittings}{F-liftability and F-splittings}}
\label{s:flift-fsplit}

In this section we study the interplay between Frobenius liftings and Frobenius splittings. On the one hand, a Frobenius lifting of a normal variety induces a Frobenius splitting. On the other, a  Frobenius splitting induces a canonical lifting of the underlying variety modulo $p^2$ (albeit without a lifting of Frobenius in general). Comparing the two constructions will allow us to descend the property of being $F$-liftable along some group quotient maps and treat $F$-liftability of products.

\subsection{\texorpdfstring{$F$-splittings associated to Frobenius liftings}{F-splittings associated to Frobenius liftings}}

A \emph{Frobenius lifting} of a $k$-scheme $X$ is, by definition, a flat scheme $\wt X$ over $W_2(k)$ lifting $X$ endowed with a morphism $\wt F \colon \wt X \to \wt X$ lifting the absolute Frobenius $F_X$. If such a pair $(\wt X, \wt F)$ exists, we call $X$ \emph{$F$-liftable}. To avoid confusion, we shall often write $\wt F_{X}$ instead of $\wt F$. An \emph{$F$-splitting} on $X$ is an $\cO_X$-linear splitting of the map $F_X^* \colon \cO_X \to F_{X*} \cO_X$, and if it exists, we call $X$ \emph{$F$-split} (the standard reference for Frobenius splittings is \cite[Chapter I]{BrionKumar}). 

\begin{example}[{\cite[\S 3.1]{PartI}}] \label{ex:flift}
The following are examples of $F$-liftable varieties:
\begin{enumerate}[(a)]
  \item ordinary abelian varieties,
  \item more generally, a variety admitting (a) as a finite \'etale cover,
  \item toric varieties,
  \item more generally, a toric fibration \cite[\S 2.1]{PartI} over an ordinary abelian variety.
\end{enumerate}
\end{example}

We extend the above terminology to morphisms as follows. We fix a Frobenius lifting $(\wt S, \wt F_S)$ of a base $k$-scheme $S$, and then by a Frobenius lifting of a scheme $X\to S$ we shall mean a Frobenius lifting $(\wt X,  \wt F_X)$ and a map $\wt X\to \wt S$ lifting $X\to S$ which is compatible with the Frobenius liftings on source and target. We then obtain a lifting $\wt F_{X/S}\colon \wt X\to \wt X{}' = \wt F_S^* \wt X$ of the relative Frobenius $F_{X/S}^* \colon X\to X' = F_S^* X$. An $F$-splitting of the map $X\to S$ (or an $F$-splitting on $X$ relative to $S$) is an $\cO_X$-linear splitting of $F_{X/S}^* \colon \cO_{X'} \to F_{X/S,*} \cO_X$.

In the above situation, and assuming that $X\to S$ is smooth, we often study the Frobenius lifting of $X/S$ by means of the morphism
\begin{equation} \label{eqn:def-xi} 
  \xi = \frac 1 p \wt F{}_{X/S}^* \colon F_{X/S}^* \Omega^1_{X'/S} \ra \Omega^1_{X/S}
\end{equation}
induced by $\wt F_{X/S}^* \colon \wt F_{X/S}^* \Omega^1_{\wt X{}'/\wt S} \to \Omega^1_{\wt X/\wt S}$, using that $F_{X/S}^* \colon F_{X/S}^* \Omega^1_{X'/S} \to \Omega^1_{X/S}$ is the zero map.

\begin{prop}[{\cite[Proof of Th\'eor\`eme 2.1 and \S 4.1]{DeligneIllusie}, \cite[Theorem 2]{BTLM}}] \label{prop:frobenius_cotangent_morphism}
  Let $(\wt S, \wt F_S)$ be a Frobenius lifting of a $k$-scheme $S$, and let $(\wt X, \wt F_X)$, $\wt X\to \wt S$ be a Frobenius lifting of a smooth morphism $X\to S$. The morphism $\xi$ defined above satisfies the following properties:
  \begin{enumerate}[(a)]
    \item The adjoint morphism $\xi^{\rm ad}\colon \Omega^1_{X'/S} \to F_{X/S*}\Omega^1_{X/S}$ has image in the subsheaf $Z^1_{X/S}$ of closed forms and provides a splitting of the short exact sequence
    \[ 
      0\ra B^1_{X/S} \ra Z^1_{X/S}\overset{C_{X/S}}\ra \Omega^1_{X'/S} \ra 0.
    \]  
    \item The Grothendieck dual of the determinant
    \[ 
      \det(\xi) \colon F_{X/S}^* \omega_{X'/S} \ra \omega_{X/S}
    \]
    is a morphism $F_{X/S,*}\cO_X \to \cO_X$ which furnishes a Frobenius splitting $\sigma$ of $X$ relative to $S$. In particular, the homomorphism $\xi$ is injective.  
  \end{enumerate}
\end{prop}

\noindent
In particular, if $S=\Spec k$ and $X$ is normal, then every Frobenius lifting of $X$ induces an $F$-splitting on $X$ by \cite[1.1.7(iii)]{BrionKumar}.

\subsection{\texorpdfstring{Canonical liftings of $F$-split schemes}{Canonical liftings of F-split schemes}}
\label{ss:f-split}

We shall now recall a functorial construction of a lifting to $W_2(k)$ of a $k$-scheme endowed with an $F$-splitting due to the third author. 

\begin{defin}
  The category $\mathbf{FSplit}$ of $F$-split schemes has as objects $F$-split schemes, i.e.\ pairs  $(X, \sigma_X)$ of a $k$-scheme $X$ and a Frobenius splitting $\sigma_X$ on $X$, and as morphisms $(X, \sigma_X)\to (Y, \sigma_Y)$ maps $f\colon X\to Y$ over $k$ for which the following diagram commutes
    \[
      \xymatrix{
        F_{Y*}f_*\cO_X \ar@{=}[r] & f_*F_{X*}\cO_X \ar[r]^-{f_*(\sigma_X)} &  f_*\cO_X \\
        F_{Y*}\cO_Y \ar[u]^-{F_{Y*}(f^*)} \ar[rr]_-{\sigma_Y} & & \cO_Y \ar[u]_-{f^*}.  
      }
    \]
\end{defin}

\begin{thm}[{\cite[Theorem 3.6]{zdanowicz}}] \label{thm:zdanowicz-canonical-lift}
  There exists a functor 
  \[ 
    (X, \sigma)\mapsto \wt X(\sigma) \colon \mathbf{FSplit} \ra \left(\text{flat schemes}/W_2(k)\right)
  \]
  together with a functorial identification $X \isom \wt X(\sigma) \otimes_{W_2(k)} k$. The structure sheaf $\cO_{\wt X(\sigma)}$ can be described as the quotient of $W_2(\cO_X)$ by the ideal $\{ (0, f)\,:\, \sigma(f)=0\}$.
\end{thm}

\subsection{\texorpdfstring{$F$-splittings and quotients by linearly reductive groups}{F-splittings and quotients by linearly reductive groups}}
\label{ss:goodquot}

For a group $G$ acting on $Y$, we say that a morphism $\pi \colon Y \to X$ is a \emph{good quotient} by $G$, if it is $G$-invariant, affine, and $\cO_X=(\pi_*\cO_Y)^G$. The following lemma will be applied for $G$ being a torus or a finite group of order prime to $p$.

\begin{lemma}\label{lem:galois_compatible_fsplit}
  Let $Y$ be an $F$-split scheme of finite type over $k$, let $G$ be a linearly reductive algebraic group acting on $Y$, and let $\pi \colon Y \to X$ be a good quotient. Then there exist splittings $\sigma_Y \colon F_{Y*}\cO_Y \to \cO_Y$ and $\sigma_X \colon F_{X*}\cO_X \to \cO_X$ such that $\pi\colon (Y, \sigma_Y)\to(X, \sigma_X)$ is a morphism of $F$-split schemes. In particular, $\pi$ admits a lifting $\swt\pi\colon \wt Y\to \wt X$. 
\end{lemma}

\begin{proof}
The relative Frobenius $F_{G/k}\colon G\to G'$ is a group homomorphism and the relative Frobenius $F_{Y/k}\colon Y\to Y'$ is $F_{G/k}$-equivariant. It follows that there is a natural linear \mbox{$G$-action} on $\Hom((F_{Y/k})_* \cO_Y, \cO_{Y'})$. Furthermore, the `evaluation at one' map
\[ 
  \varepsilon \colon \Hom((F_{Y/k})_* \cO_Y, \cO_{Y'}) \ra H^0(Y', \cO_{Y'})
\]
is a homomorphism of $G$-representations. 

If $Y$ is integral and proper, then this is a map from a finite-dimensional $G$-representation to $k$, and since $Y$ is $F$-split, this map is surjective. By the linear reductivity of $G$, this map admits a $G$-equivariant splitting, and this way we obtain a $G$-invariant Frobenius splitting of $Y$.

In general, let $V = \varepsilon^{-1}(k)$, which is the increasing union of $G$-representations $V_i$ of finite dimension over $k$, and is endowed with a~$G$-invariant map $\varepsilon\colon V\to k$. Moreover, the map $\varepsilon$ is surjective since $Y$ is Frobenius split. For some $i$, the restriction $\varepsilon|_{V_i}\colon V_i\to k$ is surjective, and since $G$ is linearly reductive, there exists a~\mbox{$G$-equivariant} splitting $s\colon k\to V_i$ of $\varepsilon|_{V_i}$. 

In particular, $Y$ admits a $G$-invariant Frobenius splitting $\sigma_Y = s(1)$. This Frobenius splitting preserves $G$-invariant sections of $\cO_Y$, and hence it descends to a Frobenius splitting on $X$. Thus $\pi\colon (Y, \sigma_Y)\to (X, \sigma_X)$ is a map in the category $\mathbf{FSplit}$. Applying the canonical lifting functor from Theorem~\ref{thm:zdanowicz-canonical-lift} yields the desired lifting $\swt\pi\colon \wt Y(\sigma_Y)\to \wt X(\sigma_X)$ of $\pi$.
\end{proof}

The method of proof of Lemma~\ref{lem:galois_compatible_fsplit} yields the following interesting result which we will not need in the sequel. 

\begin{lemma}
  Let $Y$ be an $F$-split scheme over $k$ and let $G$ be a finite group of order prime to $p$ acting on $Y$.  Then there exists a lifting $\wt Y$ of $Y$ together with a~lifting of the $G$-action.
\end{lemma}

\subsection{Uniqueness of liftings admitting a lifting of Frobenius}
\label{ss:uniq}

We have seen in Proposition~\ref{prop:frobenius_cotangent_morphism} that Frobenius liftings induce  natural Frobenius splittings.  It is natural to ask whether the lifting induced by the Frobenius splitting associated to a Frobenius lifting $(\wt X, \wt F_X)$ of a normal $k$-scheme $X$ (indicated in Theorem~\ref{thm:zdanowicz-canonical-lift}) is canonically isomorphic to $\wt X$. In fact much more is true, as indicated by the following simple but surprising result.

\begin{thm} \label{thm:fsplit-liftings}
  Let $X$ be a scheme over $k$, let $\sigma$ be a Frobenius splitting on $X$, and let $(\wt X, \wt F_X)$ be a Frobenius lifting of $X$. Let $\wt X(\sigma)$ be the canonical lifting of $X$ induced by $\sigma$ (see Theorem~\ref{thm:zdanowicz-canonical-lift}). Then one has a canonical isomorphism $\wt X(\sigma)\isomto \wt X$ (depending on $\wt F_X$ and~$\sigma$) of liftings of $X$. 
\end{thm}

\begin{proof}
The canonical lifting $\wt X(\sigma)$  comes with a natural closed immersion $i_\sigma\colon\wt X(\sigma)\to W_2(X)$. On the other hand, the Frobenius lifting $(\wt X, \wt F_X)$ induces a map 
\[ 
  \nu_{\wt X, \wt F_X} \colon W_2(X)\ra \wt X, \quad \nu^* (\wt f) = (\tilde f \text { mod } p, \delta(\wt f))
\]
where $\wt F{}^*_X(\wt f) = \wt f^p + p\cdot\delta(\wt f)$. The composition $\nu_{\wt X, \wt F_X}\circ i_\sigma \colon \wt X(\sigma)\ra \wt X$ restricts to the identity on $X$, and hence is an isomorphism of liftings of $\wt X$. This is the required map.
\end{proof}

The proof shows that $\nu_{\wt X, \wt F_X}\circ i_\sigma \colon \wt X(\sigma)\to \wt X$ fits into the commutative diagram
\[ 
  \xymatrix@R=1.5em{
      0\ar[r]  & \cO_X \ar[r] \ar[d]_{F^*_X} & \cO_{\wt X} \ar[r] \ar[d]_{\nu_{\wt X, \wt F_X}^*} & \cO_X\ar[r] \ar@{=}[d] & 0.\\
    0\ar[r] & F_* \cO_X \ar[d]_{\sigma} \ar[r] & W_2(\cO_X) \ar[r]\ar[d]_{ i_\sigma^*} & \cO_X \ar[r] \ar@{=}[d] & 0\\
    0\ar[r] & \cO_X \ar[r] & \cO_{\wt X(\sigma)} \ar[r] & \cO_X\ar[r] & 0,
  }
\]
where the left composition $\sigma \circ F_X^*$ is the identity.

\begin{cor} \label{cor:unique-f-lift}
  Let $X$ be a normal $k$-scheme, and let $(\wt X{}^{(i)}, \wt F{}_X^{(i)})$ for $i=1,2$ be two Frobenius liftings of $X$. Then $\wt X{}^{(1)} \isom \wt X{}^{(2)}$.
\end{cor}

\begin{remark}
If $X$ is smooth, then Corollary~\ref{cor:unique-f-lift} can be seen more directly as follows. The application of $\Hom(\Omega^1_X, -)$ to the short exact sequence
\[
  0\ra \cO_X \ra F_{X*} \cO_X \ra B^1_X\ra 0
\]
yields a connecting homomorphism $\delta \colon \Hom(\Omega^1_X, B^1_X)\to \Ext^1(\Omega^1_X, \cO_X)$. The forgetful map
\begin{equation} \label{eqn:forget-frob-lift}
   (\wt X, \wt F_X) \mapsto \wt X \colon
   \quad
   \{\text{Frobenius liftings of }X\}/\,\text{isom.} \ra \{\text{liftings of }X\}/\,\text{isom.}, 
\end{equation}
is a map from a torsor under $\Hom(\Omega^1_X, B^1_X)$ to a torsor under $\Ext^1(\Omega^1_X, \cO_X)$ which is equivariant with respect to the map $\delta$. If $X$ is $F$-liftable, it is $F$-split, and hence $\delta=0$. Thus \eqref{eqn:forget-frob-lift} is constant.
\end{remark}

\subsection{Descending and lifting Frobenius liftings} 

The results of \S\ref{ss:goodquot}--\ref{ss:uniq} allow us to extend \cite[Theorem~3.3.6]{PartI} in the following way.

\begin{thm}[Descending Frobenius liftability, {\cite[Theorem~3.3.6]{PartI}}] \label{thm:descending-frob-lift}
  Let $\pi \colon Y\to X$ be a morphism of schemes (essentially) of finite type over $k$ and let $(\wt Y, \wt F_{Y})$ be a Frobenius lifting of $Y$.
  \begin{enumerate}[(a)]
    \item Suppose that $\pi$ admits a lifting $\swt \pi\colon \wt Y\to \wt X$, and that one of the following conditions is satisfied:
    \begin{enumerate}[i.]
        \item $\pi^*\colon \cO_X \to R\pi_* \cO_Y$ is a split monomorphism in the derived category,
        \item $\pi$ is finite flat of degree prime to $p$,
        \item $Y$ satisfies condition $S_2$ and $\pi$ is an open immersion such that $X\setminus Y$ has codimension $>1$ in $X$.
    \end{enumerate}
    Then $F_X$ lifts to $\wt X$.  
    \item Suppose that one of the following conditions is satisfied:
      \begin{enumerate}[i.]
        \item $\cO_X\isomto \pi_* \cO_Y$ and $R^1\pi_* \cO_Y = 0$,
        \item $X$ and $Y$ are smooth and $\pi$ is proper and birational,
        \item $Y$ satisfies condition $S_3$ and $\pi$ is an open immersion such that $X\setminus Y$ has codimension $>2$ in $X$.
      \end{enumerate}
    Then there exists a unique pair of a Frobenius lifting $(\wt X, \wt F_X)$ of $X$ and a lifting $\swt \pi \colon \wt Y\to \wt X$ of $\pi$ such that $\wt F_X\circ \swt \pi = \swt \pi\circ \wt F_Y$. 
    \item Suppose that $Y$ is normal and that $\pi \colon Y\to X=Y/G$ is a good quotient by an action of a linearly reductive group $G$ on $Y$. Then there exists a lifting $\swt \pi \colon \wt Y\to \wt X$ of $\pi$ and a lifting $\wt F_X$ of $F_X$ to $\wt X$. 
  \end{enumerate}
\end{thm}

In fact, conditions (a.ii) and (b.ii) imply (a.i) and (b.i), respectively. We do not expect $\wt F_X\circ \swt \pi = \swt \pi\circ \wt F_Y$ to hold in general in situations (a) and (c).

\begin{proof}
Parts (a) and (b) were proven in \cite[Theorem~3.3.6]{PartI}. For (c), note first that since $Y$ is normal and $F$-liftable, it is $F$-split. Lemma~\ref{lem:galois_compatible_fsplit} provides compatible Frobenius splittings $\sigma_Y$ and $\sigma_X$ and a lifting $\swt\pi\colon\wt Y(\sigma_Y)\to \wt X(\sigma_X)$ of $\pi$. By Theorem~\ref{thm:fsplit-liftings}, $\wt Y(\sigma_Y)\isom \wt Y$, and we set $\wt X = \wt X(\sigma_X)$, obtaining the required lifting of $\pi$. By the  definition of a good quotient, $\cO_X = (\pi_* \cO_Y)^G$ and $\pi$ is affine, in particular $R^i \pi_* \cO_Y = 0$ for $i>0$. Since $G$ is linearly reductive, $\cO_X = (\pi_* \cO_Y)^G \to \pi_* \cO_Y$ splits, and hence assumption (a.i) is satisfied.  We apply part (a) to conclude.
\end{proof}

\subsection{Products}

Using uniqueness of liftings admitting a Frobenius lifting developed above, we can also strenghten \cite[Corollary~3.3.7]{PartI} in the following way.

\begin{cor}[{\cite[Corollary~3.3.7]{PartI}}] \label{cor:flift-products} 
  Let $X$ and $Y$ be smooth and proper schemes over $k$.  Then $X\times Y$ is \mbox{$F$-liftable} if and only if $X$ and $Y$ are. Moreover, every Frobenius lifting of $X\times Y$ arises as a product of Frobenius liftings of the factors.
\end{cor}

\begin{proof}
  The first assertion is \cite[Corollary~3.3.7]{PartI}. For the second part, we take a Frobenius lifting $(\widetilde{X \times Y},\wt F_{X \times Y})$. We already know that $X$~and $Y$ admit Frobenius liftings $(\wt X, \wt F_X)$ and $(\wt Y, \wt F_Y)$, and their product $(\wt X\times \wt Y, \wt F_X\times \wt F_Y)$ is another Frobenius lifting of $X\times Y$. But by Corollary~\ref{cor:unique-f-lift}, we must have 
\[
  \widetilde{X \times Y} \isom \wt X\times \wt Y.
\] 
Using \cite[Proposition~3.3.1]{PartI} we see that the space of Frobenius liftings on $\widetilde{X \times Y}$ is a torsor under $H^0(X\times Y, F^*T_{X \times Y})$.  The last group equals $H^0(X, F^* T_X) \oplus H^0(Y, F^*T_Y)$ which can be identified with the space of Frobenius liftings of the factors $X$ and $Y$.  This finishes the proof.
\end{proof}


\section{Frobenius liftings compatible with a divisor or a closed subscheme}
\label{s:froblift-compatible}

From the point of view of applications, it is necessary to consider Frobenius liftings of nc pairs. Recall that an nc (normal crossing) pair is a pair $(X, D)$ consisting of a smooth scheme $X$ and a divisor $D\subseteq X$ with normal crossings. A \emph{lifting} of $(X, D)$ over $W_2(k)$ is an nc pair $(\wt X, \wt D)$ where $\wt X$ is a lifting of $X$ and $\wt D\subseteq \wt X$ is an embedded deformation of $D$. Note that the requirement that $\wt D$ has normal crossings relative to $W_2(k)$ is not vacuous. By definition, a Frobenius lifting $(\wt X, \wt D, \wt F_X)$ of $(X, D)$ consists of a lifting $(\wt X, \wt D)$ of $(X, D)$ and a lifting $\wt F_X$ of $F_X$ to $\wt X$ satisfying $\wt F{}_X^*\wt D = p\wt D$. In this case, we shall say that $\wt F_X$ is \emph{compatible} with $\wt D$.

In the above situation, we can define a logarithmic variant of \eqref{eqn:def-xi}
\begin{equation} \label{eqn:def-log-xi}
  \xi = \frac 1 p \wt F{}_X^* \colon F_{X}^*\Omega^1_{X}(\log D) \to \Omega^1_{X}(\log D)
  \quad \text{(cf.\ \cite[\S 4.2]{DeligneIllusie}).}
\end{equation}

\begin{lemma} \label{lem:fsplit-comp-with-d}
  Let $(\wt X, \wt D, \wt F_X)$ be a Frobenius lifting of an nc pair $(X,D)$ over $k$. Then the $F$-splitting $\sigma$ on $X$ associated to the Frobenius lifting $(\wt X, \wt F_X)$ is compatible with the divisor $D$, i.e.\ $\sigma(F_* \cI_D)\subseteq \cI_D$ where $\cI_D\subseteq \cO_X$ is the ideal of $D$.
\end{lemma}

\begin{proof}
The maps \eqref{eqn:def-log-xi} and \eqref{eqn:def-xi} fit inside a commutative square 
\[ 
  \xymatrix{
    & F_X^* \omega_X \ar[r] \ar[d] & \omega_X \ar[d] \\
    (F_X^* \omega_X)(pD) \ar@{=}[r] & F_X^* (\omega_X(D)) \ar[r] & \omega_X(D). 
  }
\]
In particular, the top map vanishes to order $p-1$ along $D$. We conclude by \cite[1.3.11]{BrionKumar}.
\end{proof}

\begin{lemma} \label{lem:flift-on-compatible-divisors}
  Let $(X,D)$ be an nc pair over $k$, and let $(\wt X, \wt D, \wt F_X)$ be a Frobenius lifting of $(X,D)$.  Let $D_1, \ldots, D_r$ be the irreducible components of $D$. Then for every $i=1, \ldots, r$, the Frobenius lifting $\wt F_X$ induces a Frobenius lifting $\wt F_{D_i}$ on $\wt D_i$ which is compatible with the divisor $(\bigcup_{j\neq i} \wt D_j)\cap \wt D_i \subseteq \wt D_i$.
\end{lemma}

\begin{proof}
The question is local so we can assume that $X = \Spec R$ and $\wt{X} = \Spec \wt{R}$. Moreover, $D_j$ is the zero locus of $f_j \in R$, and $\wt{D}_j$ is the zero locus of $\tilde{f}_j \in \wt{R}$ where $1 \leq j \leq r$. Since $\wt{F}_X (\tilde f_j) = \swt u_j{\tilde f}_j^p$ for every $1 \leq j \leq r$ and some $\swt u_j \in \wt R{}^\times$,  we get an induced morphism $\wt{F}_{D_i} \colon \wt{D}_i \to \wt{D}_i$, where $\wt{D}_i = \Spec \wt{R}/\tilde{f}_i$, such that $\wt{F}_{D_i}(\tilde{f}_j) = \swt u_j \tilde{f}_j^p$ for $j\neq i$.
\end{proof}

Let $o_{X,D} \in \Ext^1(\Omega^1_{X}(\log D),B^1_X)$ be the obstruction to the existence of a Frobenius lifting $(\wt X, \wt D, \wt F_X)$ of a simple normal crossing pair $(X,D)$ (see \cite[Variant~3.3.2]{PartI}).  Consider the short exact sequence
\[
  0 \ra \Omega^1_{X} \ra \Omega^1_{X}(\log D) \ra \bigoplus_i \cO_{D_i} \ra 0,
\]
where $D = \sum D_i$ is a decomposition into irreducible components.  Analyzing the construction in \cite[Variant~3.3.2]{PartI} we see that the induced natural morphism 
\begin{equation} \label{eq:connecting_obstructions}
  \Ext^1\left(\Omega^1_{X}(\log D),B^1_X\right) \longrightarrow \Ext^1\left(\Omega^1_{X},B^1_X\right)
\end{equation}
maps $o_{X,D}$ into the obstruction $o_X$ to the existence of a Frobenius lifting of $X$. Moreover, if $o_{X,D}=0$, then, after fixing a Frobenius lifting of $(X,D)$, we can identify $\Hom\left(\Omega^1_{X},B^1_X\right)$ with the space of Frobenius liftings $(\wt X, \wt F_X)$ (see \cite[Proposition~3.3.1]{PartI}). With that, the natural morphism
\begin{equation} \label{eq:connecting_torsors_lifts}
  \Hom\left(\Omega^1_{X}(\log D),B^1_X\right) \ra \Hom\left(\Omega^1_{X},B^1_X\right)
\end{equation}
maps Frobenius liftings $(\wt X, \wt D, \wt F_X)$ to Frobenius liftings $(\wt X, \wt F_X)$.

\begin{lemma}\label{lem:negative_divisors} 
  Let $(X,D)$ be a simple normal crossing pair over $k$ such that 
  \[ 
    H^0(D_i,\cO_{D_i}(mD_i)) = 0 \quad  \text{for } 1 \leq m \leq p 
  \]
  and all irreducible components $D_i$ of $D$. Let $(\wt X, \wt F_X)$ be a Frobenius lifting of $X$. Then there exists a lifting $\wt D$ of $D$ with which $\wt F_X$ is compatible.
\end{lemma}

In particular, if $Y = \Bl_Z X$ is a blow-up of a smooth $k$-scheme $X$ in a smooth center $Z \subset X$, then the unique embedded deformation $\wt E \subset \wt Y$ of the exceptional divisor $E$ is preserved by every Frobenius lifting of $\wt Y$.

\begin{proof}
Let $o_X$ and $o_{X,D}$ be the obstruction classes as above. Since $o_X=0$ and $o_{X,D}$ is mapped to it by (\ref{eq:connecting_obstructions}), we get that $o_{X,D}$ is the image of an element of $\Ext^1(\bigoplus \cO_{D_i},B^1_X)$.  
Moreover, if this cohomology group vanishes, then (\ref{eq:connecting_torsors_lifts}) is surjective. Hence, in order to prove the lemma, it is enough to show that this cohomology group is zero.  To this end, we apply the local to global spectral sequence to see that 
\[
  \Ext^1\left(\bigoplus \cO_{D_i},B^1_X\right) \isom H^0\left(X,\cExt^1\left(\bigoplus \cO_{D_i},B^1_X\right)\right).
\] 
Using the short exact sequence
\[
  0 \ra \cO_X(-D_i) \ra \cO_X \ra \cO_{D_i} \ra 0
\]
we compute $\cExt^1(\bigoplus_{i=1}^r \cO_{D_i},B^1_X)$ as the cokernel of the mapping:
\[
  \bigoplus \cHom\left(\cO_X,B^1_X\right) \ra \bigoplus \cHom\left(\cO_X(-D_i),B^1_X\right),
\]
which is equal to $\bigoplus_{i=1}^r B^1_X(D_i)|_{D_i}$, because $B^1_X$ is locally free. Since $X$ is $F$-split (see Proposition~\ref{prop:frobenius_cotangent_morphism}), we have 
\[
  B^1_X\left(D_i\right)|_{D_i} 
  \subseteq \big(\cO_X(D_i) \otimes F_{X *} \cO_X\big)|_{D_i} 
  = F_{X*}(\cO_{pD_i}(pD_i)).
\] 
Here we used that $(F_{X *} \cO_X)|_{D_i}$ is the cokernel of 
\[ 
  F_{X*}(\cO_X(-pD_i)) = \cO_X(-D_i) \otimes F_{X*} \cO_X \hookrightarrow F_{X *} \cO_X,
\] 
and so it is equal to $F_{X*} \cO_{pD_i}$. 

Hence, it is enough to show that $H^0(pD_i, \cO_{pD_i}(pD_i))=0$. This follows by inductively looking at the global sections in the short exact sequences
\[
  0 \ra \cO_{D_i}((p- m+1)D_i) \ra \cO _{mD_i}(pD_i) \ra \cO_{(m-1)D_i}(pD_i) \ra 0
  \quad 2 \leq m \leq p.  \qedhere
\]
\end{proof}

\begin{defin} \label{def:flift-compatible}
  Let $X$ be a $k$-scheme, let $Z\subseteq X$ be a closed subscheme, let $(\wt X, \wt F_X)$ be a Frobenius lifting of $X$, and let $\wt Z\subseteq \wt X$ be an embedded deformation of $Z$. We say that $\wt F_X$ is \emph{compatible} with $\wt Z$ if 
  \[ 
    \wt F_X^{\#}(\cI_{\wt Z}) \subseteq \cI_{\wt Z}^p,
  \]
  or, in other words, if the image of $\wt F{}_X^* (\cI_{\wt Z})\to \cO_{\wt X}$ is contained in $\cI_{\wt Z}^p$.
\end{defin}

In particular, if $X$ is smooth and $D\subseteq X$ is a divisor with normal crossings, then $\wt F_X$ is compatible with a lifting $\wt D$ with relative normal crossings if and only if it is compatible with $\wt D$ as a closed subscheme in the sense of the above definition. In fact, if $\wt D\subseteq\wt X$ is an embedded deformation of $D$ with which $\wt F_X$ is compatible, then $\wt D$ automatically has relative normal crossings, see Corollary~\ref{cor:unique-f-lift}.  

We shall now discuss $F$-liftability of blow-ups. Let $X$ be a smooth scheme over $k$, let $Z\subseteq X$ be a smooth closed subscheme of codimension $>1$, let $\pi\colon Y\to X$ be the blowing-up of $X$ along $Z$, and let $E={\rm Exc}(\pi)$ be the exceptional divisor. In the following, we relate $F$-liftability of $X$ and of $Y$.

\begin{lemma}[{\cite[Proposition~2.2]{liedtke_satriano}}] \label{lem:liedtke_satriano_blowup}
  In the above situation, the natural transformations of deformation functors
  \[ 
    \Def_{X, Z} \ra \Def_{Y, E} \ra \Def_Y
  \]
  are isomorphisms. The composition $\Def_Y\isom\Def_{X, Z}\to \Def_X$ is given by the association
  \[
  A \in {\rm Art}_{W(k)}(k) \quad \Def_Y(A) \ni (Y,\cO_{\wt Y}) \mapsto (X,\cO_{\wt X}) = (X,\pi_*\cO_{\wt Y}) \in \Def_X(A),
  \]
  where we identify a deformation of a scheme with a thickening of its structure sheaf supported on the same topological space.
\end{lemma}

Let now $\wt Y$ be a lifting of $Y$ over $\wt S$. By Lemma~\ref{lem:liedtke_satriano_blowup}, there exist unique liftings of $\wt E$, $\wt X$, and $\wt Z$ fitting inside a commutative square
  \[
    \xymatrix{
      \wt E \ar[r] \ar[d] & \wt Y \ar[d]^{\swt \pi} \\
      \wt Z \ar[r] & \wt X.
    }
  \]
Moreover, by Theorem~\ref{thm:descending-frob-lift}(b), the map $\swt\pi$ induces an injection
  \[ 
    \left\{ \text{liftings of }F_Y\text{ to }\wt Y \right\} 
    \quad \ra \quad 
    \left\{ \text{liftings of }F_X\text{ to }\wt X \right\}.
  \]

\begin{prop} \label{prop:blow-ups}
  A lifting $\wt F_X$ of $F_X$ to $\wt X$ is in the image of the above map (that is, $\wt F_X$ extends to $\wt Y$) if and only if it is compatible with $\wt Z$ in the sense of Definition~\ref{def:flift-compatible}.
\end{prop}

\begin{proof}
Let $\wt F_X$ be a lifting of $F_X$ to $\wt X$. Suppose that $\wt F_X$ is compatible with $\wt Z$. If $\cI$ (resp.\ $\wtcI$) is the ideal sheaf of $Z$ (resp.\ $\wt Z$), then $Y = \Proj \bigoplus_{n\geq 0} \cI^n =  \Proj \bigoplus_{n\geq 0} \cI^{np}$ and $\wt Y = \Proj \bigoplus_{n\geq 0} \wtcI^n=\Proj \bigoplus_{n\geq 0} \wtcI^{np}$. The relative Frobenius $F_{Y/X}$ corresponds to the map of graded $\cO_X$-algebras
\begin{align}
  F_{Y/X}^*\colon F_X^* (\bigoplus_{n\geq 0} \cI^{n} ) \ra \bigoplus_{n\geq 0} (\cI^n)^{[p]} \ra \bigoplus_{n\geq 0} \cI^{np} \label{eq:inclusion_rel_frob}
\end{align}
induced by the inclusion. Since by assumption $\wt F{}_X^* \wtcI$ maps into $\wtcI^p$, we see that $\wt F{}_X^* \wtcI^n$ maps into $\wtcI^{np}$ for all $n\geq 0$, and we can define a map of graded $\cO_{\wt X}$-algebras
\begin{align}
  \wt F{}_{Y/X}^*\colon \wt F{}_X^* ( \bigoplus_{n\geq 0} \wtcI^{n} ) = \bigoplus_{n\geq 0} \wt F{}_X^* (\wtcI^n) \ra \bigoplus_{n\geq 0} \wtcI^{np} \label{eq:inclusion_rel_frob2}
\end{align}
again induced by the inclusion. We claim that the rational map induced by $\smash{\wt F{}^*_{Y/X}}$ is defined everywhere, and is the identity on the complement of $\wt Z$.  It suffices to show that there are no relevant homogeneous prime ideals whose preimage under \eqref{eq:inclusion_rel_frob2} becomes irrelevant.  To see this, we first observe that this property holds for the inclusion \eqref{eq:inclusion_rel_frob}, since it induces a~well-defined morphism $F_{Y/X}$.  The inclusion \eqref{eq:inclusion_rel_frob2} is only a nilpotent extension of \eqref{eq:inclusion_rel_frob}, and hence its action on homogeneous ideals is the same.  This finishes the proof of the claim.  The composition of $\smash{\wt F_{Y/X}}$ with the projection $\wt Y{}' \to \wt Y$ gives the desired extension of $\wt F_X$ to $\wt Y$.  We note that in this part of the proof the smoothness assumptions were not needed.

Now suppose that $\wt F_X$ extends to $\wt Y$. The question whether $\wt F_X$ is compatible with $\wt Z$ is local on $\wt X$, so we can assume that $\wt X = \Spec \wt A$ and $\wtcI = (\tilde x_1, \ldots, \tilde x_c)$ for some $\tilde x_i\in A$ such that $(p, \tilde x_1, \ldots, \tilde x_c)$ is a regular sequence and $c>1$. Write $\wt F{}_X^* \tilde x_i = \tilde x_i^p + pf_i$ for $f_i\in A=\wt A\otimes k$. The condition that $\wt F{}_X^*$ extends to the open subset $\wt Y_i = \Spec \wt A[\wtcI/\tilde x_i] \subseteq \wt Y$ is equivalent to the condition 
\[
\wt F{}_X^*(\tilde x_j/\tilde x_i) = \frac{\tilde x_j^p + pf_j}{\tilde x_i^p + pf_i} = 
\frac{(\tilde x_j^p + pf_j)(\tilde x_i^p - pf_i)}{\tilde x_i^{2p}} \in \wt A[\wtcI/\tilde x_i]
\]
which amounts to saying that
\[ 
  x_i^p f_j - x_j^p f_i \in \cI^{2p} \quad \text{for all }j\neq i,
\]
where $x_i\in A$ are the images of $\tilde x_i$. These equations imply that
\[ 
  f_i \in (\cI^{2p} + (x_i^p):(x_j^p)) = \cI^p \quad \text{for } i \neq j. \qedhere
\]
\end{proof}

In the following proposition, we call a morphism of $k$-schemes $f \colon X \to Y$ \emph{separable} if it is a composition of a generically smooth morphism and a closed immersion. The proof uses some results proved in the subsequent two sections.

\begin{prop} \label{prop:compatible_with_frobenius_on_products_preliminary}
  Let $(\wt X,\wt F_X)$ and $(\wt Y,\wt F_Y)$ be Frobenius liftings of smooth and proper \mbox{$k$-schemes} $X$ and $Y$.  Let $V \subseteq X \times Y$ be an integral subscheme such that one of the projections $\pi_X$ or $\pi_Y$ is separable when restricted to $V$.  Suppose that the lifting of Frobenius $\wt F_X \times \wt F_Y$ on $\wt X \times \wt Y$ is compatible with a lifting of $V$. Then $V = V_X \times V_Y$ for some integral subschemes $V_X \subseteq X$ and $V_Y \subseteq Y$. 
\end{prop}

\begin{proof}
We set $V_X$ (resp.\ $V_Y$) to be the image of $V$ under the projection $\pi_X \colon X \times Y \to X$ (resp.\ $\pi_Y \colon X \times Y \to Y$), and assume without loss of generality that $\pi_Y$ is separable when restricted to $V$.  We claim that the closed immersion $V \hookrightarrow V_X \times V_Y$ is an isomorphism.  To see this, we apply Corollary~\ref{cor:froblift-speck}(c) to the projection 
\[
  \pi_{\wt Y} \colon (\wt X \times \wt Y,\wt F_X \times \wt F_Y) \ra (\wt Y,\wt F_Y)
\]
compatible with the respective Frobenius liftings, and observe that for every $y \in V_Y$ the Frobenius lifting $\wt F_X$ is compatible with a lifting of the subscheme $V_y = V \cap (X \times \{y\})$, when interpreted as a subscheme of $X$.  By the assumptions the projection $\pi_Y$ is separable and therefore for a general $y$ the subscheme $V_y$ is a union of integral subschemes.  Using Corollary~\ref{cor:frobenius_subscheme_containment} we now notice that there are only finitely many choices for $V_y$, and therefore they are all isomorphic since $V_Y$ is connected.  This implies that the immersion $V \hookrightarrow V_X \times V_Y$ is an isomorphism on the fibers of the projection to $V_Y$ and hence is an isomorphism.
\end{proof}


\section{\texorpdfstring{$F$-splittings associated to Frobenius liftings --- relative case}{F-splittings associated to Frobenius liftings --- relative case}}
\label{sec:f_liftings_and_splittings_relative}

\subsection{\texorpdfstring{Divisors associated to $F$-splittings}{Divisors associated to F-splittings}}

We now turn to a more detailed study of relative Frobenius splittings. The Frobenius trace map $\mathrm{Tr}_{X/S} \colon F_{X/S *} \omega_{X/S} \to \omega_{X'/S}$ plays a fundamental role in the theory of $F$-splittings.

\begin{prop} \label{prop:fsplit-and-tracefsplit} 
  Let $X \to S$ be a smooth morphism of $k$-schemes. Then Grothendieck duality interchanges splittings of $\mathrm{Tr}_{X/S}$ and relative $F$-splittings of $X/S$.
\end{prop}

\begin{proof}
Clear from functoriality of Grothendieck duality for the finite flat morphism $F_{X/S}$.
\end{proof}

The following auxiliary base change result shows in particular that the fibers of a smooth relatively $F$-split morphism are $F$-split.

\begin{lemma}[cf.\ {\cite[Lemma 2.18]{PSZ}}] \label{lem:fsplit-basechange} 
  Consider a Cartesian diagram of $k$-schemes
  \[
    \xymatrix{
      W \ar[r]^{\psi} \ar[d]_{f'} & X \ar[d]^f \\
      Z \ar[r] & S,
    }
  \]
  with $f \colon X \to S$ smooth, and let $\delta_{X/S}$ be a splitting of $\mathrm{Tr}_{X/S}$. Then there exists a splitting $\delta_{W/Z}$ of $\mathrm{Tr}_{W/Z}$ induced by the following commutative diagram.
  \[
    \xymatrix{
      \psi^*  \omega_{X^{'}/S} \ar[d]_{\psi^*\delta_{X/S}}  \ar[r] ^{\sim} & \omega_{W^{'}/Z} \ar[d]^{\delta_{W/Z}} \\
      \psi^* (F_{X/S})_* \omega_{X/S}  \ar[r]^{\sim} & (F_{W/Z})_*\omega_{W/Z},
    }
  \]
 where $X'$ and $W'$ are the Frobenius twists of $X$ and $W$ relative to $S$ and $Z$, respectively.
\end{lemma}

\begin{proof}
The arrow $\psi^* (F_{X/S})_* \omega_{X/S}  \to (F_{W/Z})_*\omega_{W/Z}$ in the above diagram comes from the cohomological base change for the diagram
\[
  \xymatrix{
    W \ar[r] \ar[d]_{F_{W/Z}} & X \ar[d]^{F_{X/S}} \\
    W^{'} \ar[r] & X^{'}.
  }
\]
To conclude the proof of the lemma it is enough to show that this arrow is an isomorphism, and this is clear because $F_{X/S *}\omega_{X/S}$ is a vector bundle ($f$ is smooth, so $F_{X/S}$ is finite and flat) and the above diagram is Cartesian.
\end{proof}

The following result shows that to every $F$-splitting we can associate a $\QQ$-divisor.

\begin{prop} [\cite{schwede09} and \cite{PSZ}] \label{prop:FSPlit-DivPair} 
Let $f \colon X \to S$ be a smooth morphism of \mbox{$k$-schemes}. Then to every splitting $\delta_{X/S}$ of ${\rm Tr}_{X/S}$  we can canonically associate an effective $\QQ$-divisor $\Delta_{\delta_{X/S}}$ on $X$ such that
\begin{enumerate}[(a)]
 \item $\Delta_{\delta_{X/S}} \sim_{\QQ} -K_{X/S}$,
 \item if $f$ has connected fibers, then $\Delta_{\delta_{X/S}}$ is horizontal, i.e., it does not contain any fiber.
\end{enumerate}
\end{prop}
\begin{proof}
By adjunction, $\delta_{X/S}$ induces a morphism $F_{X/S}^*  \omega_{X^{'}/S} \to \omega_{X/S}$ which is equivalent to  $\cO_X \to \smash{\omega_{X/S}^{1-p}}$, and so we get a divisor 
\[
D_{\delta_{X/S}} \sim (1-p)K_{X/S}.
\]
Set $\Delta_{\delta_{X/S}} = \frac{1}{p-1}D_{\delta_{X/S}}$. The restriction of this $\QQ$-divisor to a fiber of $f$ is non-zero as it  corresponds to a splitting of the Frobenius trace map on this fiber (see Lemma~\ref{lem:fsplit-basechange}).
\end{proof}

For a smooth $X/k$ and a splitting $\delta_X$ of ${\rm Tr}_{X/k}$, we denote the corresponding \mbox{$\QQ$-divisor} by $\Delta_{\delta_{X}}$.

\begin{remark} \label{rem:fsplitdivisor-cartesian} In the setting of Lemma~\ref{lem:fsplit-basechange}, we have $\psi^* \Delta_{\delta_{X/S}} = \Delta_{\delta_{W/Z}}$ (see the commutative diagram in the statement of this lemma).
\end{remark} 

\begin{remark} \label{rem:coefficients-fsplit}
It is easy to see that the coefficients of $\Delta_{\delta_{X/S}}$ are at most one. When $S=\Spec k$ (which is the only case in which we will apply this observation), this follows from \cite[Theorem 3.3]{hw02} (cf.\ \cite[Theorem 4.4]{schwedesmith10}). 
\end{remark}  

\begin{remark} 
It is not necessary to assume that $f \colon X \to S$ is smooth in order to be able to associate a $\QQ$-divisor to a relative $F$-splitting. A far more general setting is described in \cite{PSZ}. When $S = \Spec k$, then a $\QQ$-divisor can be associated to a Frobenius splitting on every normal variety $X$ by an extension from the smooth locus.
\end{remark}

\subsection{Divisors associated to Frobenius liftings}
\label{ss:flift-fsplit}

To every Frobenius lifting $(\wt{X}, \wt{F}_X)$ of a smooth (or just normal) $k$-scheme $X$ we can associate a corresponding $F$-splitting $\smash{\sigma_{\wt{F}_X}}$ on $X$ and to every smooth morphism $\tilde{f} \colon \wt{Y} \to \wt{X}$ commuting with the liftings of Frobenius on $\wt Y$ and $\wt X$ we can associate a relative $F$-splitting $\smash{\sigma_{\wt{F}_{Y/X}}}$ of $Y/X$ (see Proposition~\ref{prop:frobenius_cotangent_morphism}). Here $\smash{\wt{F}_{Y/X}}$ denotes the induced lifting of the relative Frobenius.

This provides us with $\QQ$-divisors $\smash{\Delta_{\wt F_Y}}$, $\smash{\Delta_{\wt{F}_X}}$, and $\smash{\Delta_{\wt{F}_{Y/X}}}$ as in Proposition~\ref{prop:FSPlit-DivPair}. Let $f \colon Y \to X$ be the reduction of $\tilde{f}$ modulo $p$.

\begin{lemma} \label{lem:compatibility-divisors} 
  Let $X$ and $Y$ be smooth $k$-schemes. Then $\Delta_{\wt{F}_Y} = \Delta_{\wt{F}_{Y/X}} + f^*\Delta_{\wt{F}_X}$, and if $f$~has connected fibers, then $\Delta_{\wt{F}_{Y/X}}$ (resp.\ $f^*\Delta_{\wt{F}_X}$) is horizontal (resp.\ vertical).
\end{lemma}

\begin{proof}
Let $Y'$ be the base change of $Y$ along $F_X$. By the construction of the map $\xi$, we get the following commutative diagram
\[
  \xymatrix{
    0 \ar[r] & f^* F_X^* \Omega^1_X  \ar[d]^{f^*\xi_X} \ar[r]  & F_Y^* \Omega^1_Y \ar[d]^{\xi_Y} \ar[r] &  F_{Y/X}^* \Omega^1_{Y'/X} \ar[d]^{\xi_{Y/X}} \ar[r] &  0 \\    
    0 \ar[r] & f^*\Omega^1_X \ar[r]  & \Omega^1_Y \ar[r]  & \Omega^1_{Y/X} \ar[r] & 0,
  }
\]
where $F_{Y/X}^* \Omega^1_{Y'/X} \isom F_Y^* \Omega^1_{Y/X}$ and $f^*F_X^*\Omega^1_X \isom F_Y^* f^*\Omega^1_X$.  The first part of the lemma follows since the considered $\QQ$-divisors multiplied by $p-1$ are equal to ${\rm div}\, (\det \xi_Y)$, ${\rm div}\, (\det \xi_{Y/X})$, and ${\rm div}\, (\det \xi_{X})$, respectively. If $f$ has connected fibers, then $\Delta_{\wt{F}_{Y/X}}$ is horizontal by Proposition~\ref{prop:FSPlit-DivPair}. 
\end{proof}

The following corollary lists all the properties of $\Delta_{\wt{F}_X}$ we need in this article. Given a flat morphism $f \colon Y \to X$ of normal varieties such that $f_* \cO_Y = \cO_X$ and a $\QQ$-divisor $D$ on $Y$, we denote by $D^h$ and $D^v$ the horizontal and the vertical part, respectively.  

\begin{cor} \label{cor:all-properties-of-divisor} 
  Let $(\wt{Y}, \wt{F}_Y)$ be a Frobenius lifting of a smooth $k$-scheme $Y$.
  \begin{enumerate}[(a)]
    \item If $D \subseteq Y$ is a smooth irreducible divisor such that $H^0(D,\cO_D(mD))=0$ for ${1 \leq m \leq p}$, then $D \leq \Delta_{\wt{F}_Y}$.
    \item In the situation of Theorem~\ref{thm:descending-frob-lift}(b.ii), we have $\Delta_{\wt{F}_X} = \pi_*\Delta_{\wt{F}_Y}$.
    \item In the situation of Theorem~\ref{thm:descending-frob-lift}(b.i), assume that $\pi \colon Y \to X$ is smooth and let $\wt F_{Y/X}$ be the induced lifting of the relative Frobenius.  Then $\Delta_{\wt{F}_Y}^h$ is the \mbox{$\QQ$-divisor} associated to the relative $F$-splitting $\sigma_{\wt F_{Y/X}}$. In particular,
    \begin{align*}
      \Delta_{\wt{F}_Y}^h &\sim_{\QQ} -K_{Y/X}, \text{and }\\
      \Delta_{\wt{F}_Y}^v &= \pi^*\Delta_{\wt{F}_X} \sim_{\QQ} -\pi^*K_{X}.
    \end{align*}
  \end{enumerate}
\end{cor}

\begin{proof}
Statement (a) follows from Lemma~\ref{lem:negative_divisors}. Indeed, we have the morphism \eqref{eqn:def-log-xi}
\[
  \xi_{(X,D)} \colon F_{X/k}^*\Omega^1_{X^{'}}(\log D') \ra \Omega^1_X(\log D) 
\]
such that ${\rm div}(\det \xi_X) = {\rm div}(\det \xi_{(X,D)}) + (p-1)D$.

Statement (b) is clear by the construction of $\Delta_{\wt{F}_Y}$ since $g({\rm Exc}\, g)$ has codimension at least two. Statement (c) follows from Lemma~\ref{lem:compatibility-divisors}. 
\end{proof} 

We now relate certain conditions on the compatibility of subschemes for Frobenius liftings and Frobenius splittings.

\begin{lemma} \label{lem:comp_f-liftable_comp_f-split}
  Let $(\wt X,\wt F_X)$ be a Frobenius lifting of a smooth $k$-scheme $X$.  Suppose that $\wt F_X$ is compatible with a lifting of an integral subscheme $Z \subset X$.  Then the associated Frobenius splitting is compatible with $Z$
\end{lemma}

\begin{proof}
Let $\sigma_X \colon F_{X*}\cO_X \to \cO_X$ be the Frobenius splitting associated with $\wt F_X$.  First, we consider the case when $Z$ is smooth.  By Proposition~\ref{prop:blow-ups} we see that $Y = {\rm Bl}_Z X$ admits a~Frobenius lifting $(\wt Y,\wt F_Y)$ compatible with the unique lifting of the exceptional divisor $E$ and equipped with a lifting $\swt \pi \colon (\wt Y,\wt F_Y) \to (\wt X,\wt F_X)$ of the contraction morphism $\pi \colon Y \to X$.  By Lemma~\ref{lem:fsplit-comp-with-d} we see that $E$ is compatible with the Frobenius splitting $\sigma_Y \colon F_{Y*}\cO_Y \to \cO_Y$ of $Y$ induced by $\wt F_Y$.  By \cite[Lemma 1.1.8(ii)]{BrionKumar} we therefore see that $Z = \pi(E)$ is compatible with the push-forward of $\sigma_Y$ under $\pi$, which is equal to $\sigma_X$.  This finishes the proof for $Z$ smooth.  For an arbitrary integral $Z$, we just observe that the condition of being compatibly split can be checked at the generic point.
\end{proof}

\begin{cor} \label{cor:frobenius_subscheme_containment}
  Let $(\wt X,\wt F_X)$ be a Frobenius lifting of a finite type smooth $k$-scheme $X$.  Then there are only finitely integral subschemes $Z$ such that $\wt F_X$ is compatible with a~lifting of $Z$.  
\end{cor}

\begin{proof}
By Lemma~\ref{lem:comp_f-liftable_comp_f-split} we observe that every subscheme compatible with $\wt F_X$ is compatible with the associated $F$-splitting.  Then we conclude by \cite[Theorem 5.8]{schwede09}, which states that there are only finitely many subschemes compatible with a given Frobenius splitting.
\end{proof}


\section{Base change of a lifting of Frobenius}
\label{s:frob-base-change}

In this subsection, we show that a morphism from a $W_2(k)$-liftable scheme to an $F$-liftable scheme lifts to $W_2(k)$ after composing with the Frobenius. Moreover, if the source is endowed with a lifting of Frobenius, this lifting commutes with the lifting of Frobenius. Although we do not use it much in the sequel, we regard the result as essential for a good understanding of Frobenius liftings.

\begin{prop} \label{prop:frob-lift}
  Let $(\wt Y, \wt F_Y)$ be a Frobenius lifting of a $k$-scheme $Y$. Let $\phi \colon Z\to Y$ be a~morphism of $k$-schemes and let $\wt Z$ be a lifting of $Z$ over $W_2(k)$. Then there exists a~morphism $\psi \colon \wt Z\to \wt Y$ (canonically defined by \eqref{eqn:def-psi} below) such that $\psi|_Z = F_Y\circ \phi$: 
  \[
    \xymatrix{
      \wt Z \ar@/^1.5em/@{-->}[rr]^{\psi} & \wt Y \ar[r]_{\wt F_Y} & \wt Y \\
      Z \ar[u] \ar[r]_{\phi} & Y \ar[u] \ar[r]_{F_Y} & Y. \ar[u]
    }
  \]  
  If $\wt F_Z$ is a lifting of $F_Z$ to $\wt Z$, then $\psi\circ \wt F_Z = \wt F_Y\circ \psi$. Further, if $\swt\phi\colon\wt Z\to \wt Y$ is any lifting of $\phi$, then $\psi = \wt F_Y\circ \swt\phi$.
\end{prop}

Recall that if $\wt X$ is any $W_2(k)$-scheme and $X=\wt X\otimes_{W_2(k)} k$, then there is a canonical affine morphism $\theta_{\wt X} \colon \wt X \to W_2(X)$ defined on functions by the formula
\[ 
  \theta^*_{\wt X}(f_0, f_1) = (\tilde f_0)^p + p\tilde f_1
\]
where $\tilde f_0, \tilde f_1\in \cO_{\wt X}$ are any liftings of $f_0, f_1\in \cO_X$. Moreover, if $\wt X$ is flat over $W_2(k)$, then a lifting $\wt F_X\colon \wt X\to \wt X$ of the absolute Frobenius $F_X\colon X\to X$ induces an affine morphism $\nu_{\wt X, \wt F_X} \colon W_2(X) \to \wt X$ in the opposite direction, defined on functions by
\begin{equation} \label{eqn:def-nu-map} 
  \nu^*_{\wt X, \wt F_X}(\tilde f) = \left(f, \delta(f)\right) \quad \text{ for } \tilde f\in \cO_{\wt X},
\end{equation}
where $f$ the image of $\tilde f$ in $\cO_X$ and $\delta(f)$ is the unique element such that $\wt F{}^*_X(\tilde f) = \tilde f^p  + p\cdot \delta(f)$.  We can recover $\wt F_X$ from $\nu_{\wt X, \wt F_X}$ by the formula $\wt F_X = \nu_{\wt X, \wt F_X}\circ\theta_{\wt X}$, while the other composition $\theta_{\wt X}\circ\nu_{\wt X, \wt F}$ coincides with the Witt vector Frobenius $W_2(F_X)$:
\begin{equation} \label{eqn:recover-f}
  \xymatrix{
    W_2(X) \ar[r]_-{\nu_{\wt X, \wt F}} \ar@/_1.6em/[rr]_{W_2(F_X)} & \wt{X} \ar[r]_-{\theta_{\wt X}} \ar@/^1.6em/[rr]^{\wt F} & W_2(X) \ar[r]_-{\nu_{\wt X, \wt F}} & \wt X.
  }
\end{equation}

\begin{proof}[Proof of Proposition~\ref{prop:frob-lift}]
We define the desired lifting $\psi \colon \wt Z\to \wt Y$ as the composition
\begin{equation} \label{eqn:def-psi}
  \wt Z \xrightarrow{\theta_{\wt Z}} W_2(Z) \xrightarrow{W_2(\phi)} W_2(Y) \xrightarrow{\nu_{\wt Y, \wt F_Y}} \wt Y.
\end{equation}
Explicitly, on functions, $\psi$ takes the form
\[ 
  \psi^*(\tilde f) = \widetilde{\phi^*(f)}^p + p \cdot \phi^*(\delta_{\wt F_Y}(\tilde f)),
\]
where $\widetilde{\phi^*(f)}\in\cO_{\wt Z}$ is a local section mapping to $\phi^*(f)\in \cO_Z$. It is thus clear that $\psi|_Z = F_Y\circ \phi$.

We now check that $\psi\circ \wt F_Z = \wt F_Y\circ \psi$ when $\wt Z$ is flat over $W_2(k)$ and endowed with a~Frobenius lifting $\wt F_Z$. This follows from the commutativity of the following diagram:
\[ 
  \xymatrix{
    \wt Z \ar@/^1.5em/[rr]^-{\wt F_{Z}} \ar[r]_-{\theta_{\wt Z}} & W_2(Z) \ar[r]_-{\nu_{\wt Z, \wt F_Z}} \ar[d]_{W_2(\phi)} & \wt Z \ar[r]^-{\theta_{\wt Z}} & W_2(Z) \ar[d]^{W_2(\phi)} \\
     & W_2(Y) \ar[r]_-{\nu_{\wt Y, \wt F_Y}} & \wt Y \ar@/_1.5em/[rr]_-{\wt F_{Y}} \ar[r]^-{\theta_{\wt Y}} & W_2(Y) \ar[r]^-{\nu_{\wt Y, \wt F_Y}} & \wt Y.
  }
\] 
Indeed, by \eqref{eqn:recover-f} the composition $\wt Z\to W_2(Z)\to \wt Z$ (resp.\ $\wt Y\to W_2(Y)\to \wt Y$) equals $\wt F_Z$ (resp.\ $\wt F_Y$).  To show the commutativity of the diagram, we note that the compositions $W_2(Z)\to \wt Z\to W_2(Z)$ and $W_2(Y)\to \wt Y\to W_2(Y)$ are the Witt vector Frobenius morphisms $W_2(F_Z)$ and $W_2(F_Y)$. These are functorial, which shows that the middle square (and hence the whole diagram) commutes. 

The final assertion follows from the commutativity of the following diagram:
\[ 
  \xymatrix{
    \wt Z \ar[d]_{\swt\phi} \ar[r]^-{\theta_{\wt Z}}  & W_2(Z) \ar[d]^{W_2(\phi)} & \\
    \wt Y \ar[r]^-{\theta_{\wt Y}} \ar@/_1.5em/[rrr]_-{\wt F_{Y}} & W_2(Y) \ar[rr]^-{\nu_{\wt Y, \wt F_Y}} & & \wt Y.
  }
\]
Here the square commutes by functoriality of the maps $\theta$.
\end{proof}

\begin{cor} \label{cor:lifts-of-Cartesian-diagrams}
  Let $f\colon X\to Y$ and $\phi\colon Z\to Y$ be morphisms of $k$-schemes, and let $(\wt X, \wt F_X)$, $(\wt Y, \wt F_Y)$, and $(\wt Z, \wt F_Z)$ be Frobenius liftings of $X$, $Y$, and $Z$, respectively. Let $\tilde f\colon\wt X\to \wt Y$ be a lifting of $f$ commuting with the Frobenius liftings, and let $\psi\colon\wt Z\to \wt Y$ be the lifting of $F_Y\circ \phi$ given by Proposition~\ref{prop:frob-lift}. Form the cartesian diagram
  \begin{equation} \label{eqn:frob-bc}
    \xymatrix{
      \wt W \ar[r] \ar[d] & \wt X \ar[d]^{\tilde f} \\
      \wt Z \ar[r]_{\psi} & \wt Y.
    }
  \end{equation}
  Then $\wt W$ admits a Frobenius lifting $\wt F_W$ such that the maps $\wt W\to \wt Z$ and $\wt W\to  \wt X$ commute with the Frobenius liftings.  Moreover, for every subscheme  $\wt V \subseteq \wt X$ compatible with $\wt F_X$, its preimage under $\wt W \to \wt X$ is compatible with $\wt F_W$. 
\end{cor}

Setting $Z=\Spec k$ (in which case $F_Z$ is an isomorphism), we obtain the following.

\begin{cor} \label{cor:froblift-speck}
Let $(\wt Y, \wt F_Y)$ be a Frobenius lifting of a $k$-scheme $Y$.
\begin{enumerate}[(a)]
  \item The construction of Proposition~\ref{prop:frob-lift} with $Z=\Spec k$ yields a section of the specialization map
  \[ 
    \wt Y(W_2(k)) \ra Y(k).
  \]
  In particular, every $k$-point of $Y$ lifts (canonically) to a point of $\wt Y$.
  \item Let $(\wt X, \wt F_X)$ be a Frobenius lifting of a $k$-scheme $X$, and let $\tilde f\colon \wt X\to \wt Y$ be a map commuting with the Frobenius liftings. Then for every $y\in Y(k)$, the fiber $X_y = f^{-1}(y)$ is $F$-liftable.
  \item In the situation of (b), for every subscheme $Z \subset X$ such that there exists a lifting $\wt Z \subset \wt X$ compatible with $\wt F_X$, the Frobenius lifting of $X_y$ is compatible with the induced lifting of the intersection $f^{-1}(y) \cap Z$.
\end{enumerate}
\end{cor}

\begin{example}
The scheme $\wt Y = \{xy=p\}\subseteq \bb{A}^2_{W_2(k)}$ does not admit a lifting of Frobenius (not even locally), because the $k$-point $(0,0)\in Y = \{xy=0\}$ does not admit a lifting modulo~$p^2$. Of course $Y$ admits a Frobenius lifting 
\[ 
  \wt Y{}'=\{xy=0\}\subseteq \bb{A}^2_{W_2(k)}, 
  \quad \wt F'^{*}_Y(x)=x^p, \quad \wt F'^{*}_Y(y)=y^p.
\]
Recall that by Corollary~\ref{cor:unique-f-lift}, since $Y$ is $F$-split, it admits at most one lifting to which $F_Y$ lifts.  
\end{example}


\section{\texorpdfstring{$F$-liftability of surfaces}{F-liftability of surfaces}}
\label{s:surfaces}

The goal of this section is to show Conjecture~\ref{conj:froblift} for smooth surfaces. Let $(\wt X, \wt F_X)$ be a Frobenius lifting of a smooth $k$-scheme $X$. As in \S\ref{ss:flift-fsplit}, we have the associated effective $\QQ$-divisor 
\[ 
  \Delta_{\wt{F}_X} = \frac{1}{p-1}{\rm div}(\det(\xi_{\wt F_X})) \ \sim_{\QQ} \ -K_X.
\]
A~careful analysis of $\Delta_{\wt{F}_X}$ plays a vital role in this section.  We further define $D_{\wt{F}_X} = \lfloor \Delta_{\wt{F}_X} \rfloor$. Note that $D_{\wt{F}_X}$ is reduced (see Remark~\ref{rem:coefficients-fsplit}). If $X$ is a toric variety with its standard Frobenius lifting, then $\smash{\Delta_{\wt F_X} = D_{\wt F_X}}$ is the toric boundary (the complement of the open orbit) of~$X$.

First, we tackle the case of rational surfaces.

\begin{lemma} \label{lemma:frobenius_blow-ups}  
  Let $X$ be a smooth surface over $k$, and let 
  \[
  \pi\colon Y=\Bl_x X\ra X
  \]
   be the blow-up of $X$ at a closed point $x\in X$. Let $(\wt X, \wt F_X)$ and $(\wt Y, \wt F_Y)$ be Frobenius liftings of $X$ and $Y$, and let $\swt \pi \colon \wt X\to \wt Y$ be a lifting of $\pi$ satisfying $\wt F_X \circ \swt \pi = \swt \pi \circ \wt F_Y$. Suppose that ${\rm Supp}\, \Delta_{\wt F_X}$ has simple normal crossings at $x$. Then $x \in \Sing D_{\wt{F}_X}$.
\end{lemma}

Readers familiar with the language of birational geometry may notice that this is a direct consequence of the fact that $x$ is a log canonical center of $(X, \smash{\Delta_{\wt{F}_X}})$ by Corollary~\ref{cor:all-properties-of-divisor}(a) (in fact, this shows that the above result is valid in higher dimensions for blow-ups along arbitrary smooth centers). We provide a more elementary explanation below.

\begin{proof}
For the exceptional divisor $E = {\rm Exc}(\pi)$, we have $E \leq \Delta_{\wt{F}_Y}$ by Corollary~\ref{cor:all-properties-of-divisor}(a). Since $\pi_*\Delta_{\wt{F}_Y} = \Delta_{\wt{F}_X}$ (see Corollary~\ref{cor:all-properties-of-divisor}(b)), the support of $\pi^*\Delta_{\wt{F}_X} - \Delta_{\wt{F}_Y}$ is exceptional. By definition, this $\QQ$-divisor is linearly equivalent to $K_{Y/X}$, which is equal to $E$ by \cite[Proposition V.3.3]{hartshorne77}, and so 
\[
  \pi^*\Delta_{\wt{F}_X} - \Delta_{\wt{F}_Y} = E.
\]
In particular, $\pi^*\Delta_{\wt{F}_X} = 2E + \Delta'$, where $E \not \subseteq {\rm Supp}\, \Delta'$. As ${\rm Supp}\, \Delta_{\wt{F}_X}$ has simple normal crossings at $x$ and the coefficients of $\Delta_{\wt{F}_X}$ are at most one (see Remark~\ref{rem:coefficients-fsplit}), this is only possible if $x \in \Sing D_{\wt F_X}$.
\end{proof}

\begin{remark} \label{remark:toric_blow-ups} 
Let $f \colon (Y, D_Y) \to (X, D_X)$ be a toric morphism between two-dimensional toric pairs, which on the level of schemes is a blowing-up of a smooth point $x \in X$. Then $x \in \Sing(D_X)$ and $D_Y = f_*^{-1}D_X + \mathrm{Exc}(f)$. Moreover, the converse is also true, that is the blowing-up of a smooth toric surface at toric points is a toric morphism.
\end{remark} 

Let us call a pair $(X,D)$ of a normal variety $X$ and a reduced effective divisor $D$ on $X$ \emph{sub-toric} if $X$ admits the structure of a toric variety such that $D$ is invariant under the torus action. If $D$ is the maximal invariant divisor, then we call $(X,D)$ toric (cf.\ \cite[\S 2.1]{PartI}). 

\begin{lemma} \label{lemma:frobenius_on_hirzebruch} 
  Let $\FF_n=\PP_{\PP^1}(\cO_{\PP^1}\oplus\cO_{\PP^1}(n))$ be the $n^{\rm th}$ Hirzebruch surface for $n\geq 0$. Then for every Frobenius lifting $(\wt \FF_n, \wt F)$ of $\FF_n$, the pair $(\FF_n, D_{\wt{F}})$ is sub-toric. Moreover, if $n=0$, then $\mathrm{Supp}\, \Delta_{\wt{F}}$ has simple normal crossings, and if $n>0$, then $\mathrm{Supp}\, \Delta_{\wt{F}}$ has simple normal crossings at every point of the negative section $C \subseteq \FF_n$.
\end{lemma}

Write $\FF_n = \PP_{\PP^1}(E)$ for $E = \cO_{\PP^1}\oplus\cO_{\PP^1}(n)$. A choice of the splitting $E \isom \cO_{\PP^1}\oplus\cO_{\PP^1}(n)$ provides $\FF_n$ with a natural toric structure for which the natural morphisms $\PP_{\PP^1}(\cO_{\PP^1}) \to \FF_n$ and $\PP_{\PP^1}(\cO_{\PP^1}(n)) \to \FF_n$ are toric. Thus $(\FF_n, D)$ is a toric pair if and only if

\begin{itemize}
  \item for $n=0$, we have $D = G_1 + G'_1 + G_2 + G'_2$, where $G_1$, $G'_1$, and $G_2$, $G'_2$ are distinct fibers of the two projections $\pi_1, \pi_2 \colon \FF_0 \to \PP^1$.
  \item for $n>0$, we have $D = C + C' + G_1 + G_2$, where $C$ is the unique negative section (corresponding to $\PP_{\PP^1}(\cO_{\PP^1}) \to \FF_n$), $C'$ is a section disjoint from $C$, and $G_1$, $G_2$ are two distinct fibers of the projection $\pi \colon\FF_n \to \PP^1$. 
\end{itemize}

\begin{proof}
Let us fix a Frobenius lifting $(\wt \FF_n, \wt F)$. If $n=0$, then, by Corollary~\ref{cor:flift-products}, we have $\Delta_{\wt{F}} = \pi_1^* \Delta_1 + \pi_2^*\Delta_2$, where $\Delta_1$ and $\Delta_2$ are effective $\QQ$-divisors on $\PP^1$, and so $\mathrm{Supp}\, \Delta_{\wt{F}}$ is simple normal crossing. Since $\omega_{\PP^1 \times \PP^1} \isom \cO_{\PP^1 \times \PP^1}(-2,-2)$, both $\lfloor \Delta_1 \rfloor$ and $\lfloor \Delta_2 \rfloor$ consist of at most two irreducible divisors, which concludes the proof.

If $n>0$, then by Theorem~\ref{thm:descending-frob-lift}(b.i) we get a compatible Frobenius lifting of $\PP^1$ and so Corollary~\ref{cor:all-properties-of-divisor}(c) implies that
\begin{align*}
  \Delta^{h} &\sim_{\QQ} -K_{\FF_n / \PP^1}, \\
  \Delta^{v} &\sim_{\QQ} -\pi^*K_{\PP^1},
\end{align*} 
where $\Delta^{h}$ and $\Delta^{v}$ are the horizontal and the vertical part of $\smash{\Delta_{\wt F}}$, respectively. Moreover, Corollary~\ref{cor:all-properties-of-divisor}(a) gives $\Delta^{h} = \Delta' + C$, where $\Delta'$ is an effective $\QQ$-divisor such that ${C \not \subseteq \mathrm{Supp}\, \Delta'}$. 

As $K_{\FF_n} + \Delta^{v} + \Delta' + C \sim_{\QQ} 0$, we have
\[
  \Delta' \cdot C = -(K_{\FF_n} + C + \Delta^v ) \cdot C = 2 - \Delta^v \cdot C = 0
\]
by adjunction, and so $\Delta'$ is disjoint from $C$. Given that ${\rm Supp}\, (\Delta^v + C)$ is simple normal crossing,  so is ${\rm Supp}\, \Delta_{\wt{F}}$ along $C$. Moreover, for a fiber $G$ of $\pi$
\[
  \Delta' \cdot G = -(K_{\FF_n} + \Delta^v + C) \cdot G = -(K_{\FF_n} + G)\cdot G -(\Delta^v + C) \cdot G = 1,
\]
by adjunction as $G^2=0$, and so $\lfloor \Delta' \rfloor$ is zero or is a single section disjoint from $C$. Since $\lfloor \Delta^v \rfloor$ consists of at most two distinct fibers, this concludes the proof of the lemma.
\end{proof}

\begin{prop} \label{prop:f-lifts_on_rational_surfaces} 
  Let $Y$ be a smooth projective rational surface. If $Y$ is $F$-liftable, then it is toric.
\end{prop}

\begin{proof}
Since $\PP^2$ is toric, we can assume that $Y \not \isom \PP^2$. Every smooth rational surface which is not isomorphic to $\PP^2$ admits a birational morphism to a Hirzebruch surface $\pi\colon Y\to \FF_n$ for some $n\geq 0$, which factors into a sequence of monoidal transformations
\[ 
  Y = X_m \xrightarrow{\pi_{m-1}} X_{m-1} \ra  \cdots \ra X_1 \xrightarrow{\pi_0} X_0 = \FF_n, \quad X_{i+1} = \Bl_{x_i} X_i.
\]
We assume that $n$ is \emph{minimal} among such. It follows that if $n>0$, then $\pi(\mathrm{Exc}(\pi)) \subseteq C$, where $C\subseteq \FF_n$ is the negative section. Indeed, the blow-up $\Bl_{x} \FF_n$ at any $x \not \in C$ admits a~morphism to $\FF_{n-1}$ constructed by contracting the strict transform of the fiber through $x$ of the natural projection $\FF_n \to \PP^1$.

Let $(\wt Y, \wt F_Y)$ be a Frobenius lifting of $Y=X_m$. By Theorem~\ref{thm:descending-frob-lift}(b), there exist Frobenius liftings $(\wt X_i, \wt F_{i})$ of $X_i$ for $0 \leq i \leq m$, and liftings $\swt \pi_i\colon \wt X_{i+1}\to \wt X_i$ such that $\wt F_{i}\circ \swt \pi_i = \swt \pi_i \circ \wt F_{i+1}$. By Lemma~\ref{lemma:frobenius_on_hirzebruch} we know that $(X_0, D_{\wt{F}_0})$ is sub-toric and $\Delta_{\wt{F}_0}$ has simple normal crossings at $\pi(\mathrm{Exc}(\pi))$. Therefore, for every $i>0$ the $\QQ$-divisor $\Delta_{\wt F_{i}}$ has simple normal crossings at $x_i \in X_i$ (it is contained in the union of the exceptional locus of $\pi$ and the support of the strict transform of $\Delta_{\wt F_0}$ by Corollary~\ref{cor:all-properties-of-divisor}(b)).

By induction we can show that $(X_i, D_{\wt{F}_{i}})$ is sub-toric for every $0 \leq i \leq m$. Indeed, if $(X_{i-1}, D_{\wt{F}_{i-1}})$ is sub-toric, then by Remark~\ref{remark:toric_blow-ups} it is enough to show that $x_{i-1} \in \Sing D_{\wt{F}_{i-1}}$, which follows by Lemma~\ref{lemma:frobenius_blow-ups}. This concludes the proof.
\end{proof}

\begin{remark} 
In the course of the proof, we showed that if $(\wt X, \wt{F}_X)$ is a Frobenius lifting of a smooth projective rational surface $X\not\simeq \PP^2$, then $(X, \lfloor \Delta_{\wt{F}_X} \rfloor)$ is sub-toric. This is false for some liftings of Frobenius on $\PP^2$.
\end{remark} 

We now turn our attention to ruled surfaces. We say that a rank two vector bundle $E$ on a~curve $C$ is \emph{normalized} if $H^0(C,E)\neq 0$ and $H^0(C, E \otimes \cL)=0$ for every line bundle $\cL$ such that $\deg \cL < 0$. Given a ruled surface $X = \PP_C(E)$ we can assume that $E$ is normalized by replacing $E$ with $E \otimes \cL$ for some line bundle $\cL$.

\begin{prop} \label{prop:ruled_surface} 
  Let $X=\PP_C(E)$ be a smooth projective ruled surface for a normalized rank two vector bundle $E$ on an ordinary elliptic curve $C$. Then $X$ is $F$-liftable if and only if $E$ is not a non-split extension of $\cO_C$ with itself.
\end{prop}

\begin{proof}
If $E$ is decomposable (that is, a direct sum of two line bundles), then $X$ is $F$-liftable by Example~\ref{ex:flift}(d). Hence, we can assume that $E$ is indecomposable. By \cite[Theorem V.2.15]{hartshorne77}, there are only two such ruled surfaces corresponding to $E$ being a non-split extension of $\cO_C$ with $\cO_C(c)$ where $c \in C$, and $E$ being a non-split extension of $\cO_C$ with itself. 

In the former case, $X \isom \Sym^2(C)$ (see for instance \cite[Section 6]{fuentes}), and it is easy to see that it is $F$-liftable. Indeed, let $(\wt C, \wt F_C)$ be the canonical Frobenius lifting of $C$ (see Example~\ref{ex:flift}(a)), where $\wt F_C \colon \wt C \to \wt C$. Then $\wt F_C \times \wt F_C \colon \wt C \times \wt C \to \wt C \times \wt C$ is equivariant under the natural $\ZZ/2\ZZ$ action by swapping the coordinates, and so it descends to $\Sym^2(\wt C) \to \Sym^2(\wt C)$.

Therefore, we are left to show that $X$ is not $F$-liftable when $E$ is a non-split extension of $\cO_C$ with itself. By contradiction assume that it does admit a Frobenius lifting $(X, \wt F_X)$. By Lemma~\ref{lem:flat_ruled_surface}, we know that $X$ is a quotient of $Y = C' \times \PP^1$ by $\FF_p$ acting independently on $C'$ and $\PP^1$, where $C'$ is the Frobenius twist of $C$, and the action on $(x:y) \in \PP^1(k)$ is defined as $(x : y) \mapsto (x+ly : y)$  for a fixed $l \in \FF_p$. This is illustrated by the following diagram
\[
  \xymatrix{
    \PP^1 & \ar[l]_-{\rho} Y\ar[r]^-{V'} \ar[d]_{\pi'} \ar@{}[dr]|-\square & X \ar[d]^{\pi} \\
    & C' \ar[r]_-{V} & C.
  }
\]
Since $V'$ is {\'e}tale, \cite[Lemma~3.3.5]{PartI} implies that $Y$ admits a Frobenius lifting $(\wt Y, \wt F_Y)$ such that $\Delta_{\wt F_Y} = V'^* \Delta_{\wt F_X}$, where $\Delta_{\wt F_Y}$ and $\Delta_{\wt F_X}$ are the $\QQ$-divisors associated to $\wt F_Y$ and $\wt F_X$, respectively. In particular, $\Delta_{\wt F_Y}$ is $\FF_p$-invariant.

As $\Delta_{\wt F_Y} \sim_{\QQ} -K_{C'\times \PP^1} = -\rho^*K_{\PP^1}$, we have $\Delta_{\wt F_Y}=\rho^*T$ for some $\QQ$-divisor $T$ on $\PP^1$. Let $G$ be a fiber of $\pi$ over a general point $c \in C$ and $\Delta = \Delta_{\wt F_X}|_G$. Since $\Delta_{\wt F_Y}$ is $\FF_p$-invariant and it is a pullback from $\PP^1$, we get that $\Delta$ is invariant under the action of $\FF_p$ on $G \isom \PP^1$.

By Corollary~\ref{cor:all-properties-of-divisor}(c) and Remark~\ref{rem:fsplitdivisor-cartesian} applied to $\pi'$, we get that $\Delta$ is the associated $\QQ$-divisor of an $F$-splitting of $\PP^1$. This contradicts Lemma~\ref{lem:p1}.
\end{proof}

We needed the following two lemmas in the proof of the above proposition. 

\begin{lemma} \label{lem:p1} 
  There does not exist an $\FF_p$-invariant $F$-splitting of $\PP^1$, where $\FF_p$ acts on $\PP^1$ via translations, that is $(x : y) \mapsto (x + ly : y)$ for $l \in \FF_p$. 
\end{lemma}

Note that an $F$-splitting is invariant under an action of a group if and only if the corresponding $\QQ$-divisor is invariant.

\begin{proof}
Assume by contradiction that such an $F$-splitting exists and let $\Delta$ be the corresponding $\QQ$-divisor (see Proposition~\ref{prop:FSPlit-DivPair}). By the definition of $\Delta$, we get that $(p-1)\Delta$ is an effective integral divisor and $\deg \Delta = 2$. Furthermore, the coefficients of $\Delta$ are at most one (see Remark~\ref{rem:coefficients-fsplit}). Since each orbit of the action of $\FF_p$ on $\PP^1$ is of length $p$ except for the one of the fixed point $\infty \in \PP^1$, the only $\QQ$-divisor satisfying the aforementioned properties is
\[
  \Delta = \left(1 - \frac{1}{p-1}\right)(\infty) + \sum_{i=0}^{p-1}\frac{1}{p-1}(x_i),
\]
where $x_i = x_0 + i \in \bb{A}^1(k)$. Up to an action of an automorphism we can assume that $x_0 = 0$. This yields a contradiction, because $\PP^1$ cannot be $F$-split with such an associated $\QQ$-divisor $\Delta$ by \cite[Example 3.4]{cgs14} and \cite[Proposition 5.3 (2)]{schwedesmith10}. One can check that the trace of an $F$-splitting of $\PP^1$ cannot be equal to $(p-1)\Delta$  directly by noticing that $x(x-y)\cdots(x-(p-1)y)y^{p-2}$ has coefficient zero at $x^{p-1}y^{p-1}$ (see \cite[Theorem 1.3.8]{BrionKumar}). 
\end{proof}

\begin{lemma} \label{lem:flat_ruled_surface} 
  Let $C$ be an ordinary elliptic curve, and let $C'$ be the Frobenius twist of $C$. Consider the action of $\FF_p$ on $C'\times \PP^1$ with $l \in \FF_p$ acting as
  \[
    (c, (x:y)) \longmapsto (c+l\alpha, (x+ly : y)),
  \]
  for $c \in C'$ and $(x : y) \in \PP^1$, where $\alpha \in C'[p]$ is a fixed $p$-torsion point. Let $X$ be the quotient $(C'\times\PP^1)/\FF_p$. Then $X \isom\PP_C(E)$ for $E$ being a non-split extension of $\cO_C$ with itself. 
\end{lemma}

\begin{proof}
By definition we have the following Cartesian diagram
\[
  \xymatrix{
    C' \times \PP^1 \ar[r]^-{V'} \ar[d]_{\pi'} & X \ar[d]^{\pi} \\
    C' \ar[r]_-{V} & C,
  }
\]
where $V$ is the quotient by $C'[p]$. Note that $C' \times \PP^1 = \PP_{C'}(E')$, where $E'$ is a vector bundle sitting inside a split (but not $\FF_p$-equivariantly so) extension of $\FF_p$-equivariant sheaves
\[
  0 \ra \cO_{C'} \ra E' \ra \cO_{C'} \ra 0,
\]
with $E' \to \cO_{C'}$ corresponding to the $\FF_p$-equivariant section $C' \times (1:0)$ of $\pi'$. 
Therefore, $X = \PP_C(E)$ for a vector bundle $E$ on $C$ such that $V^*E$ is $\FF_p$-equivariantly isomorphic to $E'$, and the above short exact sequence descends to $C$ making $E$ a non-split extension of $\cO_C$ by itself.
\end{proof}
Having handled rational and ruled surfaces, we are ready to proceed to the general case.

\begin{thm} \label{thm:c1-surfaces} 
  Let $X$ be a smooth projective surface over $k$. Then $X$ is $F$-liftable if and only if $X$ is 
  \begin{enumerate}[(i)]
   \item an ordinary abelian surface, 
   \item a hyperelliptic surface being a quotient of a product of two ordinary elliptic curves,
   \item a ruled surface $\PP_C(E)$ for a normalized rank two vector bundle $E$ over an ordinary elliptic curve $C$ except when $E$ is a non-trivial extension of $\cO_C$ with itself, or
   \item a toric surface.
  \end{enumerate}
\end{thm}
In particular, Conjecture~\ref{conj:froblift} is true for surfaces.

\begin{proof}
If $X$ is $F$-liftable, then it is $F$-split and $\smash{\omega_X^{1-p}}$ is effective (see Proposition~\ref{prop:frobenius_cotangent_morphism}). In particular, we only need to consider the case of $\kappa(X) \leq 0$.  If $\kappa(X)=0$, then $K_X$ is \mbox{$\QQ$-effective}, and hence $\smash{\omega_X^{p-1}}$ is trivial and $X$ is minimal. In this case, the theorem follows from \cite[Theorem 1]{xin16}, but for the convenience of the reader we present a simplified argument.  

When $K_X$ is torsion, $X$ is $F$-liftable if and only if it is a quotient of an ordinary abelian surface by a free action of a finite group (see \cite[Theorem 2]{MehtaSrinivas}). By Hirzebruch--Riemann--Roch, such quotients satisfy $\chi(X,\cO_X)=0$, and so by the classification of surfaces they are abelian, hyperelliptic, or quasi-hyperelliptic. We can exclude the surfaces of the latter type as they contain rational curves (see \cite[\S 7]{LiedtkeSurfaces}): indeed, if $\phi \colon \PP^1 \to X$ is non-constant, then there exists an injection $T_{\PP^1} \to \phi^*T_X$, so $T_X$ cannot be \'etale trivializable (cf.\ \cite[Proposition 5]{xin16}). Hyperelliptic surfaces are \'etale quotients of products of elliptic curves $E_1 \times E_0$, except for case a3) which is an \'etale quotient of an abelian surface $A = E_1\times E_0/\mu_2$ (see \cite[p.\ 37]{BombieriMumford} for the notation). Therefore they are $F$-liftable if and only if $E_1$ and $E_0$ are ordinary, by \cite[Theorem 2]{MehtaSrinivas} and \cite[Lemma~3.3.5]{PartI}. Note that $A$ is ordinary if and only if $E_1 \times E_0$ is, as they are isogeneous to each other. This concludes our analysis of Kodaira dimension zero.

As of now, we assume $\kappa(X)=-\infty$. If $X$ is of type (iii) or (iv), then it is $F$-liftable by Example~\ref{ex:flift}(c) and Proposition~\ref{prop:ruled_surface}. To conclude, we assume that $X$ is $F$-liftable and show that $X$ is of type (iii) or (iv). If $X$ is rational, then it is toric by Proposition~\ref{prop:f-lifts_on_rational_surfaces}. Thus, we can assume that there exists a birational morphism $f \colon X \to \PP_C(E)$, where $\PP_C(E)$ is a~ruled surface over $C \not \isom \PP^1$. By Theorem~\ref{thm:descending-frob-lift}(b), the curve $C$ is $F$-liftable, and so it is an ordinary elliptic curve.

We claim that $X \isom \PP_C(E)$. Otherwise, $f$ factors through a monoidal transformation $\pi\colon\Bl_{x} \PP_C(E) \to \PP_C(E)$, where $x \in \PP_C(E)$. In particular, Theorem~\ref{thm:descending-frob-lift}(b) implies that $Y=\Bl_{x} \PP_C(E)$ is $F$-liftable, which is impossible. Indeed, if $Y$ is $F$-liftable, then so is $(Y, {\rm Exc}\, \pi)$ by Lemma~\ref{lem:negative_divisors}, and hence $Y$ is $F$-split compatibly with ${\rm Exc}\, \pi$ by Lemma~\ref{lem:fsplit-comp-with-d}. Therefore, we get an $F$-splitting on $C$ compatible with a point (see \cite[Lemma 1.1.8(ii)]{BrionKumar}), which is impossible as $C$ is an elliptic curve.

Since we know that $X$ is a ruled surfaces over an ordinary elliptic curve, the theorem follows from Proposition~\ref{prop:ruled_surface}.
\end{proof}

\begin{remark} \label{rem:xin} 
The $F$-liftability of minimal surfaces has been considered in \cite{xin16}. Note that our results do not agree when $\kappa(X)=-\infty$, as \cite[Theorem 1 (2b)]{xin16} claims that all ruled surfaces over ordinary elliptic curves are $F$-liftable. It seems to us that the gluing argument used at the end of the proof of Proposition~9 in \emph{op.cit.}\ is incomplete, as it is unclear why $h$ extends to a regular function over $V$.
\end{remark}


\section{Fano threefolds}
\label{s:fano3}

In this section, we will work assuming the following claim, which would follow from some assertions in the literature which we were unable to verify completely. See Appendix~A for a detailed discussion of the issue of boundedness for Fano threefolds.

\begin{assertion} \label{assn:fano3-bdd}
  There exists an integer $m>0$ such that for every Fano threefold $X$ over an algebraically closed field of arbitrary characteristic the divisor $-mK_X$ is base-point free.
\end{assertion}

\begin{thm} \label{thm:fano3}
  Assume Assertion~\ref{assn:fano3-bdd} holds true. Then there exists a $p_0$ such that for every prime $p\geq p_0$, every $F$-liftable Fano threefold over an algebraically closed field of characteristic $p$ is toric.
\end{thm}

\begin{remark} \label{rmk:fano3}
\begin{enumerate}[(a)]
  \item Our proof of Theorem~\ref{thm:fano3} is heavily based on the Mori--Mukai classification of Fano threefolds, known only in characteristic zero. Note that $F$-liftable smooth Fano varieties in positive characteristic are rigid and admit a~unique lifting to characteristic zero, since by Bott vanishing \eqref{eqn:bott-van} we have
  \[
  H^i(X, T_X) = H^i(X,\Omega^{n-1}_{X} \otimes \omega_{X}^{\vee}) = 0 \quad \text{for }i>0.  \qedhere
  \]
 
  The main reference for the Mori--Mukai classification used in this section is \cite[Tables §12.3-§12.6]{shafarevich}. We denote a Fano threefold of Picard rank $\rho$ and whose number in the tables is $n$ as M--M $\rho.n$ (for example, M--M 2.12). 

  \medskip
  
  \item In the proof, we only need Assertion~\ref{assn:fano3-bdd} to hold in characteristic $p \gg 0$, and only for $F$-liftable (and hence rigid, liftable to characteristic zero, and $F$-split) Fano threefolds. Without assuming Assertion~\ref{assn:fano3-bdd}, our proof shows that every $F$-liftable Fano threefold described by the Mori--Mukai classification is toric. It also shows that the assertion of Theorem~\ref{thm:fano3} holds (with a smaller, more explicit bound on $p_0$) if one could prove the following result.

  \medskip

  \noindent {\bf Claim.} \emph{Let $R$ be a complete discrete valuation ring and let $X, Y/R$ be smooth and proper. Suppose that both special fibers $X_0, Y_0$ are rigid Fano threefolds, and that the generic fibers $X_\eta, Y_\eta$ are isomorphic. Then $X_0$ and $Y_0$ are isomorphic.}

  \medskip

  \noindent The above statement holds trivially if $R$ contains a field, which makes it quite intriguing.

  \medskip

  \item In fact in Proposition~\ref{prop:fano-rigid-spec}, proven using boundedness statement coming from Assertion~\ref{assn:fano3-bdd}, we prove that there exists a prime number $p_0$ such every Fano variety in characteristic $p>p_0$ arises as a reduction of a characteristic zero model coming from Mori--Mukai classification.  This $p_0$ is exactly the necessary bound in Theorem~\ref{thm:fano3}.

  \medskip

  \item Combined with \cite[Theorem~3]{PartI}, Theorems~\ref{thm:c1-surfaces} and \ref{thm:fano3} show that the conjecture of Occhetta and Wiśniewski \cite{occhetta_wisniewski} (see \cite[Conjecture~2]{PartI}) holds in characteristic zero if the target $X$ is a surface or a Fano threefold, with no additional assumptions.
\end{enumerate}
\end{remark}

The following theorem summarizes everything we need to know about the classification.

\begin{thm} \label{thm:mori-mukai}
  Let $X$ be a smooth complex Fano threefold. Then either $X$ is toric, or there exists a fibration $\pi\colon X\to Y$ to a smooth Fano variety $Y$, where $Y$ is either non-rigid or one of the following varieties:
  \begin{enumerate}[(a)]

    \item[(a-1)] A smooth divisor of tri-degree $(1,1,1)$ in $\PP^1\times \PP^1\times \PP^2$ \hfill (M--M 3.17)

    \item[(a-2)] The blow-up $Y = \Bl_W Z$ where

      \begin{center}
      \begin{tabular}{r l r l r}
        $Z=$ \hspace{-1em} & $\PP^1 \times \PP^2$ & $W=$ \hspace{-1em} & a smooth curve of bi-degree $(2, 1)$ & \hphantom{alabababababa} (M--M 3.21) \\
        & $\PP^1 \times \PP^1\times \PP^1$ & &  a smooth curve of tri-degree $(1,1,2)$ & (M--M 4.3) \\
        & $\PP^1 \times \PP^1 \times \PP^1$ & & tri-diagonal curve & (M--M 4.6) \\
        & $\PP^1\times \PP^1 \times \PP^1$ & & a smooth curve of tri-degree $(0,1,1)$ & (M--M 4.8)
      \end{tabular}
      \end{center}

    \medskip

    \item[(b-1)] $Y = \Gr(2,5) \cap \PP^6$, where $\Gr(2,5) \subseteq \PP^9$ is the Grassmannian of planes in $\PP^5$ embedded in $\PP^9$ by the  Pl\"ucker embedding  and $\PP^6 \subseteq \PP^9$ is a general linear subspace. \hfill (M--M 1.15)

    \item[(b-2)] A smooth quadric in $\PP^4$. \hfill  (M--M 1.16)

    \item[(b-3)] A smooth divisor of bi-degree $(1,1)$ in $\PP^2\times \PP^2$. \hfill (M--M 2.32)

    \medskip
    \item[(c-1)] The blow-up of a smooth planar conic in $\PP^3$. \hfill (M--M 2.30)

    \item[(c-2)] The blow-up of a smooth conic contained in $\{t\}\times \PP^2$ in $\PP^1\times \PP^2$. \hfill (M--M 3.22)

    \medskip
    \item[(d)] The blow-up of the twisted cubic in $\PP^3$. \hfill (M--M 2.27)

    \medskip
    \item[(e)] A non-toric del Pezzo surface. 
  \end{enumerate}
\end{thm}

\begin{proof}
The only necessary information not immediately available from looking at the tables in \cite[Tables §12.3-§12.6]{shafarevich} is which Fano threefolds are rigid. To this end, we will use the Hirzebruch--Riemann--Roch formula 
\begin{align*} 
  \chi(X, E) &= \frac{1}{24}{\rm rk}(E) c_1(T_X)c_2(T_X) + \frac{1}{12}c_1(E)\left(c_1(T_X)^2 + c_2(T_X)\right) \\
  & + \frac{1}{4} c_1(T_X)\left(c_1(E)^2 - 2c_2(E)\right) + \frac{1}{6}\left(c_1(E)^3 - 3c_1(E)c_2(E) + 3c_3(E)\right).
\end{align*}
In particular, $\chi(X, \cO_X)=c_1(T_X)c_2(T_X)/24$; on the other hand, $\chi(X, \cO_X)=1$ by Kodaira vanishing, so $c_1(T_X)c_2(T_X)=24$. The number $c_3(T_X)$ equals the topological Euler characteristic, $2 + 2\rho(X) - b_3(X)$. We deduce the following formula
\[
  \chi(X, T_X) = h^0(X, T_X) - h^1(X, T_X) = \frac{1}{2}(-K_X)^3 - 18 + \rho(X) - \frac{1}{2}b_3(X),
\] 
where the values of the invariants on the right hand side are listed in \emph{loc.\ cit.} Whenever this expression is negative, the corresponding Fano threefold is not rigid. This applies to threefolds (M--M 1.1--14, 2.1--25, 3.1--12, and 4.1--2). 

The converse statement fails in one example which we check by hand. The threefold (M--M 2.28) is the blow-up of $\PP^3$ along a plane cubic $C \subseteq \PP^3$. By Lemma~\ref{lem:liedtke_satriano_blowup}, $\smash{\Def_{\Bl_C \PP^3}\isom \Def_{\PP^3, C}}$. Further, the forgetful transformation $\smash{\Def_{\PP^3, C}\to \Def_C}$ is non-constant. We conclude that $\Bl_C \PP^3$ is not rigid. 

With this at hand, the assertions now follow from \cite[Tables §12.3-§12.6]{shafarevich}. \emph{N.B.} The toric Fano threefolds are 2.33--2.36, 3.25--3.31, 4.9--4.12, 5.2, 5.3 are toric (cf.\ \cite[Ch. 12, Remarks (i), p.\ 216]{shafarevich}, note that toric varieties 3.25 and 4.12 are missing from this~list).
\end{proof}

\begin{lemma} \label{lemma:connected_components}
  Let $\pi \colon H \to S$ be a morphism of finite type.  Then there exists a dense open subset $U \subseteq S$ such that if $T$ is a trait with geometric generic (resp.\ closed) point $\bar\eta$ (resp.\ $\bar k$) with a morphism $T\to U$, and if $x_1, x_2\in H(T) = \Hom_S(T, H)$ are points whose images in $H(\bar\eta)$ lie in the same connected component of $H_{\bar\eta}$, then their images in $H(\bar k)$ lie in the same connected component of $H_{\bar k}$.
\end{lemma}

\begin{proof}
By considering $S\times \Spec \ZZ[1/\ell]$ for every prime $\ell$ separately, we may assume that there exists an $\ell$ invertible on $S$. By constructibility of higher direct images \cite[Th\'eor\`eme~1.9]{DeligneThFinitude} there exists a dense open $U\subseteq S$ such that the \'etale sheaves $\pi_* \FF_\ell$ are locally constant, with formation commuting with base change. Then for $T\to U$ as in the statement, the cospecialization map
\[ 
  \FF_\ell^{\pi_0(H_{\bar\eta})} = H^0(H_{\bar\eta}, \FF_\ell) = (\pi_* \FF_\ell)_{\bar\eta}
  \ra
  (\pi_* \FF_\ell)_{\bar k} = H^0(H_{\bar k}, \FF_\ell) = \FF_\ell^{\pi_0(H_{\bar k})}
\]
is an isomorphism. The required assertion follows.
\end{proof}

\begin{prop} \label{prop:fano-rigid-spec}
Assume Assertion~\ref{assn:fano3-bdd} holds true. There exists a $p_0$ such that for every trait $T$ with residue characteristic $p\geq p_0$, and every two smooth and proper $X, Y$ over $T$ such that
\begin{enumerate}
  \item the geometric special fibers $X_{\bar k}$ and $Y_{\bar k}$ are $F$-split Fano threefolds,
  \item moreover, they are rigid, i.e.\ $H^1(X_{\bar k}, T_{X_{\bar k}}) = 0 = H^1(Y_{\bar k}, T_{Y_{\bar k}})$, 
  \item the geometric generic fibers $X_{\overline{\eta}}$ and $Y_{\overline{\eta}}$ are isomorphic,
\end{enumerate}
the geometric special fibers $X_{\bar k}$ and $Y_{\bar k}$ are isomorphic as well.
\end{prop}

\begin{proof}
Let $m$ be as in Theorem~\ref{thm:boundedness}. We claim that  the $m$-th power of the anticanonical bundle is relatively very ample and yields embeddings $X, Y\hookrightarrow \PP^{N-1}_T$ with Hilbert polynomial $\chi$. Indeed, since $X_{\bar{k}}$ is an $F$-split Fano variety, $H^i(X_{\bar k}, \mathcal{O}_{X_{\bar k}}(-mK_{X_{\bar{k}}}))=0$ for $i>0$. Therefore, the restriction $H^0(X, \mathcal{O}_X(-mK_{X/T})) \to H^0(X_{\bar k}, \mathcal{O}_{X_{\bar k}}(-mK_{X_{\bar{k}}}))$ is surjective (for example by Grauert's and semicontinuity theorems), and the very ampleness of $-mK_{X_{\bar{k}}}$ implies the very ampleness of $-mK_{X/T}$. The analogous argument works for $Y$. We note that both the Hilbert polynomial and the integer $N$ are equal for the two families using condition (3), semicontinuity and Kodaira vanishing valid for $F$-split varieties.


Let $H \subseteq \mathrm{Hilb}_\chi(\PP^{N-1}_\ZZ)$ be the open subscheme parametrizing smooth three-dimensional subschemes of $\PP^{N-1}$ with ample anticanonical bundle and vanishing higher cohomology of the tangent bundle. By (1) and Akizuki--Nakano vanishing (true by (2) for $p>2$), $X$ and $Y$ give two points $[X], [Y] \in H(T)$. Assumption (3) implies that their images in $H(\bar\eta)$ lie in the same connected component (in fact, the same orbit of $\mathrm{PGL}_{N}$). Applying Lemma~\ref{lemma:connected_components} to $H\to \Spec \ZZ$, we get that if $p\gg 0$ (with bound independent of $X$ and $Y$), then $[X_{\bar k}]$ and $[Y_{\bar k}]$ lie in the same connected component of $H(\bar k)$. But since $X_{\bar k}$ is rigid, the infinitesimal neighborhood of $[X_{\bar k}]$ in $H_{\bar k}$ is contained in its $\mathrm{PGL}_{N}$-orbit. This shows that the $\mathrm{PGL}_N$-orbits on $H_{\bar k}$ are open, and therefore also closed. Thus $[X_{\bar k}]$ and $[Y_{\bar k}]$ lie in the same orbit and hence $X_{\bar k}$ and $Y_{\bar k}$ are isomorphic.
\end{proof}

\begin{thm}  \label{thm:fano3-list}
  Let $X$ be a smooth projective threefold over an algebraically closed field $k$ of characteristic $p>0$ admitting a birational map $X\to Y$ where $Y$ can be described as in the list in Theorem~\ref{thm:mori-mukai}. Then $X$ is not $F$-liftable.
\end{thm}

The above theorem implies that Conjecture~\ref{conj:froblift} holds for Fano threefolds in characteristic $p$ provided that the Mori--Mukai classification (more precisely, Theorem~\ref{thm:mori-mukai}) is valid in that characteristic. In the following proof, we deduce Theorem~\ref{thm:fano3} from it by lifting a given $F$-liftable variety to characteristic zero, applying the Mori--Mukai classification there, and then descending back to characteristic $p$. This last step is a bit delicate, and it relies on Proposition~\ref{prop:fano-rigid-spec}, which in turn requires the boundedness of Fano threefolds (Assertion~\ref{assn:fano3-bdd}).

\begin{proof}[Proof of Theorem~\ref{thm:fano3} (assuming Theorem~\ref{thm:fano3-list})]
First, we define $p_0$. Let $X_1, \ldots, X_r$ be the non-toric rigid complex Fano threefolds. By Theorem~\ref{thm:mori-mukai}, for each $i$ there exists a fibration $\pi_i \colon X_i\to Y_i$ where $Y_i$ is either non-rigid or described as in the list (a-1) \ldots (e). As each $X_i$ is rigid, it is defined over $\overline \QQ$. Further, by Lemma~\ref{lemma:models} below, each fibration $\pi_i$ is defined over $\overline\QQ$ as well. We can therefore find a number field $K$, an integer $N$, and models $\pi_i\colon \mathscr{X}_i\to \mathscr{Y}_i$ of $\pi_i\colon X_i\to Y_i$ over $\cO_K[1/N]$ with the following properties:
\begin{itemize}
  \item $\mathscr{X}_i$ and $\mathscr{Y}_i$ are smooth and proper over $\cO_K[1/N]$, with ample anticanonical bundles, 
  \item $\pi_i$ is a fibration,
  \item $H^i(\mathscr{X}_i, T_{\mathscr{X}_i/\cO_K[1/N]}) = 0$ for $i>0$, 
  \item if $Y_i$ is rigid, then $\mathscr{Y}_i$ can be described over $\cO_K[1/N]$ as in the list (a-1) \ldots (e).
\end{itemize}
Let $m$ be as in Theorem~\ref{thm:boundedness}, and let $\chi_i(t) = \chi(X_i, -tK_{X_i})$, $N_i = \chi_i(m) = h^0(X_i, -mK_{X_i})$. 
By potentially increasing $N$, we may assume using \cite[1.6.E Exercises (5)]{BrionKumar} that all the geometric closed fibres of $\cX_i \to \Spec(\cO_K[1/N])$ are $F$-split and hence the conditions (1) and (2) of Proposition~\ref{prop:fano-rigid-spec} are satisfied for the localization of $\cX_i \to \Spec(\cO_K[1/N])$ at any prime in the base.  

Let $X$ be an $F$-liftable Fano threefold over an algebraically closed field $k$ of characteristic $p\geq p_0$. Suppose that $X$ is not toric. Since $X$ is rigid and $H^2(X, T_X) = 0$ (Remark~\ref{rmk:fano3}(a)), there exists a unique deformation $\wt X$ of $X$ over $W(k)$ (which algebraizes because the ample line bundle $\omega_X^{-1}$ lifts).  Since $X$ is $F$-liftable it is also $F$-split and hence the assumpions of Proposition~\ref{prop:fano-rigid-spec} are satisfied for the family $\wt X/ W(k)$.  This property is clearly invariant under base change.  Let $F$ be an algebraically closed field containing both $\mathbf{C}$ and $W(k)$. Since $\wt X_F$ is a rigid Fano threefold, which is moreover not toric (otherwise it would have a toric Fano model $Y$ over $W(k)$, and Proposition~\ref{prop:fano-rigid-spec} would imply $X\isom Y_0$ i.e.\ $X$ toric), there exists an isomorphism $\iota\colon (X_i)_F \isom \wt X_F$ for some $i\leq r$. Our next goal is to spread it out and obtain an isomorphism over $k$.

Let $V_0$ be the completion of $\cO_K[1/N]$ at a prime above $p$. Since $V_0$ is a finite extension of $\ZZ_p$, there exists a finite extension $V$ of $W(k)$ also contained in $F$ such that $\cO_K[1/N]\subseteq V$ as subrings of $F$. Passing to a further finite extension of $V$, we can assume that the isomorphism $\iota$ is defined over the subfield $V[1/p]$ of $F$. 
%
%
Applying Proposition~\ref{prop:fano-rigid-spec} to $\wt X_V$ and $(\mathscr{X}_i)_V$ over $\Spec V$ gives an isomorphism
\[ 
  \mathscr{X}_i \otimes_{\cO_K[1/N]} k  \simeq \wt X_k = X.
\]

In particular, $\mathscr{X}_i \otimes_{\cO_K[1/N]} k $ is $F$-liftable. By Theorem~\ref{thm:descending-frob-lift}(b) and Proposition~\ref{prop:relative_Kodaira_Fsplit} applied for $L = -K_X$, so is $\mathscr{Y}_i \otimes_{\cO_K[1/N]} k$. In particular, being Fano, it must be rigid, which implies by semi-continuity that $Y_i$ is rigid as well. In this case, $\mathscr{Y}_i \otimes_{\cO_K[1/N]} k$ is described as in the list (a-1) \ldots (d), and by Theorem~\ref{thm:fano3-list} we obtain a contradiction. Therefore $X$ is toric.
\end{proof}

\begin{prop}[Relative Kodaira vanishing for $F$-split total space]
\label{prop:relative_Kodaira_Fsplit}
Let $X$ be an $F$-split smooth projective variety, let $f \colon X \to Y$ be a projective morphism and let $L$ be an $f$-ample divisor.  Then $R^jf_*\cO_X(K_X + L) = 0$ for $j>0$.
\end{prop}

\begin{proof}
Let $\mathscr{F}_j = R^jf_*\cO_X(K_X + L)$ and let $M$ be a fixed ample divisor on $Y$. Take $n$ large enough so that (1) $\mathscr{F}_j \otimes \cO_Y(nM)$ is globally generated and (2) has no higher cohomology for all $j$ and that (3) $L + n f^* M$ is ample on $X$. By (1), the vanishing of $\mathscr{F}_j$ is equivalent to $H^0(Y, \mathscr{F}_j\otimes \cO_Y(nM)) = 0$. By the projection formula, the Leray spectral sequence for $\cO_X(K_X + L + n f^* M)$ reads
\[ 
  E_2^{ij} = H^i(Y, \mathscr{F}_j \otimes \cO_Y(nM)) \quad \Rightarrow \quad H^{i+j}(X, \cO_X(K_X+L+nf^* M)),
\]
and by (2) this collapses yielding $H^0(Y, \mathscr{F}_j\otimes \cO_Y(nM)) \isom H^j(X, \cO_X(K_X+L+nf^* M))$, which is $H^j(X, \cO_X(K_X + \text{ample}))$ by (3). Since $X$ is $F$-split, it satisfies Kodaira vanishing, and hence this group is zero for $j>0$. 
\end{proof}

\begin{lemma}[{Models over subfields}] \label{lemma:models}
  Let $k\subseteq k'$ be an extension of algebraically closed fields, and let $X$ be a normal projective variety over $k$ with $H^1(X, \cO_X) = 0$. Then for every fibration $\pi'\colon X_{k'} \to Y'$ to a normal projective variety $Y'$ over $k'$ there exists a fibration $\pi\colon X\to Y$ to a normal projective variety $Y$ over $k$ and an isomorphism $\iota\colon Y_{k'} \isom Y'$ such that the following triangle commutes
  \[ 
    \xymatrix@R=1.3em{
      X_{k'} \ar[rr]^{\pi'} \ar[dr]_{\pi_{k'}} & & Y' \\
      & Y_{k'} \ar[ur]_{\iota} .
    }
  \]
\end{lemma}

\begin{proof}
Let $\cO_{Y'}(1)$ be an ample line bundle on $Y'$ and let $L' = (\pi')^* \cO_{Y'}(1)$, so that 
\[ 
  Y' = \Proj \bigoplus_{n\geq 0} \Gamma(X_{k'}, (L')^n).
\]
Since $H^1(X, \cO_X) = 0$, $\Pic X$ is discrete, and hence $\Pic X \isomto \Pic X_{k'}$. Let $L \in \Pic X$ correspond to $L'$ under this isomorphism; then $L$ is ample. We set
\[ 
  Y = \Proj \bigoplus_{n\geq 0} \Gamma(X, L^n).
\]
Since by flat base change $\Gamma(X, L^n)\otimes_k k'\isom\Gamma(X_{k'}, (L')^n)$, we get the desired isomorphism $Y_{k'}\simeq Y'$.
\end{proof}

\begin{proof}[Proof of Theorem~\ref{thm:fano3-list}]
By Theorem~\ref{thm:descending-frob-lift}, it is enough to show that $Y$ is not $F$-liftable. We do this case by case, deferring the more involved arguments to lemmas following the proof.

\medskip

\textbf{Cases (a-1) and (a-2): blow-ups on a product.}  Consider first the third example (M--M 4.6) in (a-2), i.e.\ the blow-up of $\PP^1 \times \PP^1 \times \PP^1$ along the tridiagonal curve
\[
  C = \left\{ (x,x,x) \in \PP^1 \times \PP^1 \times \PP^1 \mid x \in \PP^1 \right\}. 
\]
The variety $X$ cannot be $F$-liftable, because then $(\PP^1 \times \PP^1 \times \PP^1,C)$ would admit a Frobenius lifting by Lemma~\ref{lem:negative_divisors} and Proposition~\ref{prop:blow-ups}, which is impossible by Proposition~\ref{prop:compatible_with_frobenius_on_products_preliminary} (as the center $C$ is not the product of subvarieties of the factors). The remaining cases (M--M  3.17, 3.21, 4.3, and 4.8) are not $F$-liftable by an analogous argument. Note that Fano threefold (a-1) (M--M 3.17) is the blow-up of $\PP^1 \times \PP^2$ along the graph of the Segre embedding $\PP^1 \to \PP^2$ of degree two. 

\medskip

\textbf{Cases (b-1), (b-2), and (b-3): violating Bott vanishing \eqref{eqn:bott-van}.} Fano threefolds (b-1), (b-2), and (b-3) do not satisfy Bott vanishing (Lemma~\ref{lemma:grassmann} below, \cite[Example~3.2.6]{PartI} or \cite[\S 4.1]{BTLM},  and \cite[\S 4.2]{BTLM}, respectively), and so they are not $F$-liftable. 

\medskip

\textbf{Cases (c-1) and (c-2): inducing a lifting of Frobenius on $\PP^2$ compatible with a conic.} Let $X = \Bl_C W$ be a blow-up of either $W=\PP^3$ or $\PP^1\times \PP^2$ along a conic $C$ contained in a plane $H \simeq \PP^2 \subseteq W$, and suppose that $X$ is $F$-liftable.

Let $\overline{H} \isom \PP^2$ be the strict transform of $H$ on $X$, and let $E$ be the exceptional divisor of the blow-up. By Lemma~\ref{lem:negative_divisors}, we get that $(X, \overline{H} + E)$ admits a Frobenius lifting, and so $(\overline{H}, \overline{H}\cap E)$ is $F$-liftable by Lemma~\ref{lem:flift-on-compatible-divisors}. Since $C\isom \overline{H}\cap E$ is a conic, this contradicts Lemma~\ref{lem:conics_not_frobenius_compatible}. 

\medskip
\textbf{Case (d): blow-up along twisted cubic.} See Lemma~\ref{lem:twisted-cubic} below.

\medskip
\textbf{Case (e): del Pezzo surface.} Follows from Theorem~\ref{thm:c1-surfaces}.
\end{proof}

\begin{lemma}[M--M 1.16] \label{lemma:grassmann} 
  Let $X = \Gr(2,5) \cap \PP^6$, where $\Gr(2,5) \subseteq \PP^9$ is the Grassmannian of planes in $\PP^5$ embedded in $\PP^9$ by the  Pl\"ucker embedding  and $\PP^6 \subseteq \PP^9$ is a general linear subspace. Then $H^1(X,\Omega_X^2(1))\neq 0$, where $\cO_X(1)$ is the restriction of $\cO_{\PP^9}(1)$ to $X$.
\end{lemma}

\begin{proof}
The setting admits a natural lifting to characteristic zero, and by the semicontinuity theorem, it is enough to show that the required non-vanishing holds for this lifting. Therefore, we can assume that $X$ is defined over $\bb{C}$. 

We have $\omega_{\Gr(2,5)} \isom \cO_{\Gr(2,5)}(-5)$, hence $\omega_X \isom \cO_X(-2)$ and therefore $H^1(X,\Omega_X^2(1)) = H^1(X, T_X(-1))$. First, we prove that $H^i(X, T_{\Gr(2,5)}(-1)|_X)=0$ for all $i$. Consider the Koszul resolution
\[
  0 \ra \cO_{\Gr(2,5)}(-3) \ra \cO_{\Gr(2,5)}(-2)^{\oplus 3} \ra \cO_{\Gr(2,5)}(-1)^{\oplus 3} \ra \cO_{\Gr(2,5)} \ra \cO_X \ra 0,
\]
and tensor it by $T_{\Gr(2,5)}(-1)$ to get
\[
  0 \to T_{\Gr(2,5)}(-4) \to T_{\Gr(2,5)}(-3)^{\oplus 3} \to T_{\Gr(2,5)}(-2)^{\oplus 3} \to T_{\Gr(2,5)}(-1) \to T_{\Gr(2,5)}(-1)|_X \to 0.
\]
By \cite[Theorem, p.171, (3)]{snow-grassmann}, all the cohomology groups of $T_{\Gr(2,5)}(-k) \isom \Omega^5_{\Gr(2,5)}(5-k)$ vanish for $1 \leq k \leq 4$, and so the above exact sequence shows that the same holds for $T_{\Gr(2,5)}(-1)|_X$.

The dual of the conormal exact sequence tensored by $\cO_X(-1)$ is
\[
  0 \ra T_X(-1) \ra T_{\Gr(2,5)}(-1)|_X \ra \cO_X^{\oplus 3} \ra 0.
\]
Since the cohomology groups of the middle sheaf vanish, we get $H^1(T_X(-1)) = H^0(X,\cO_X)^{\oplus 3} \neq 0$.
\end{proof}

\begin{lemma} \label{lem:conics_not_frobenius_compatible} 
  Let $D \subseteq \PP^2$ be a smooth conic. Then $(\PP^2, D)$ is not $F$-liftable.
\end{lemma}

\begin{proof}
Assume by contradiction that $(\PP^2,D)$ is $F$-liftable and fix a Frobenius lifting. Let 
\[
  \xi \colon F^*\Omega^1_{\PP^2}(\log D) \to \Omega^1_{\PP^2}(\log D)
\]
be the associated morphism \eqref{eqn:def-log-xi}. Choose a general point $x \in D$ and take $C \subseteq \PP^2$ to be the line tangent to $D$ at $x$. Since $x$ is general and $\xi$ is generically an isomorphism (see Proposition~\ref{prop:frobenius_cotangent_morphism}), we get that $\xi|_{C}$ is injective. In particular, if $\Omega^1_{\PP^2}(\log D)|_C = \cO_C(a) \oplus \cO_C(b)$ for some $a,b \in \ZZ$, then $a,b \leq 0$ (see the proof of \cite[Lemma~6.2.1(a)]{PartI}). Indeed, if $a\leq b$, then the injectivity of $\xi|_C$ implies that we have a nonzero map $\cO_C(pb)\to \cO_C(c)$ where $c\in \{a,b\}$, so $pb\leq c\leq b$ and $b\leq 0$.

We will show that  $\Omega^1_{\PP^2}(\log D)|_C = \cO_C(-2) \oplus \cO_C(1)$ yielding a contradiction (cf.\ \cite[Lemma~4]{xin16}). To this end, choose a standard affine chart on $\bb{A}^2 \subseteq \PP^2$ in which $C$ is described as $y=0$, and $D$ as $y-x^2=0$.  In the chart $\bb{A}^2 \cap C$, we get that $\Omega^1_{\PP^2}(\log D)|_C$ is generated by 
\[
  dx \ \  \text{ and } \ \ \frac{d(y-x^2)}{y-x^2} = -\frac{dy}{x^2} + \frac{2dx}{x}.
\]
Now, we take the chart of $\PP^2$ having coordinates $\frac{1}{x}$ and $\frac{y}{x}$. In this chart $C$ and $D$ are disjoint, and so in the restriction of this chart to $C$  the vector bundle $\Omega^1_{\PP^2}(\log D)|_C$ is generated by
\[
   d\left(\frac{1}{x}\right) = -\frac{1}{x^2}dx
   \quad \text{ and } \quad
  d\left(\frac{y}{x}\right) = \frac{dy}{x} = 2dx + x\left(\frac{dy}{x^2} - \frac{2dx}{x}\right).
\]
Therewith, the coordinate-change matrix is:
\[
  \begin{bmatrix}
-\frac{1}{x^2} & 0\\
  2 & -x
  \end{bmatrix},
\]
and so $\Omega^1_{\PP^2}(\log D)|_C = \cO_C(-2) \oplus \cO_C(1)$.
\end{proof}

A surprising property of $F$-liftability is that if $X$ is $F$-liftable and $f \colon Y \to X$ is a smooth morphism such that $Rf_* \cO_Y = \cO_X$, then $f$ is relatively $F$-split. This is not true in general, if we assume that $X$ is only $F$-split.

\begin{lemma}[M--M 2.27] \label{lem:twisted-cubic} 
  Let $C \subset \PP^3$ be the twisted cubic, that is,  the image of $\PP^1$ under the embedding given by the full linear system $|\cO_{\PP^1}(3)|$. Then the blow-up $Y = \Bl_{C} \PP^3$ is not $F$-liftable.
\end{lemma}

\begin{proof}
Let $E$ be the exceptional divisor of the blow-up $\pi \colon Y \to \PP^3$. Firstly, we observe that by \cite[Application 1, page 299]{szurek_wisniewski_Fano} $Y$ is isomorphic to a projective bundle $\PP(\cE)$ over $\PP^2$ for some non-split rank two vector bundle $\cE$. The morphism $f \colon Y  \to \PP^2$ is given by a pencil of quadrics in $\PP^3$ containing $C$.

If $Y$ is $F$-liftable, then Corollary~\ref{cor:all-properties-of-divisor}(a) and (c) imply that  $-(K_{Y / \PP^2}+E)$ is $\QQ$-linearly equivalent to an effective $\QQ$-divisor. In what follows, we show that this is not true. By the construction of $f$, we have $Q = f^{-1}(L)$ for some line $L \subseteq \PP^2$, where $Q$ is the strict transform of a quadric in $\PP^3$ containing $C$. Since $\omega_Y \isom \cO_Y(E) \otimes \pi^*\cO_{\PP^3}(-4)$, we have $K_{Y} \sim -2Q - E$, and so:
\[
K_{Y / \PP^2}+E \sim -2Q + 3f^{-1}(L) \sim Q,
\]
which is not $\QQ$-linearly equivalent to an anti-effective $\QQ$-divisor.
\end{proof}

\begin{remark} 
Using Corollary~\ref{cor:lifts-of-Cartesian-diagrams} and a more intricate version of the above argument one can show that a rank two vector bundle $\cE$ on $\PP^n$ is decomposable if and only if $\PP_{\PP^n}(\cE)$ is $F$-liftable. This seems intriguing from the viewpoint of Hartshorne's conjecture predicting that all such vector bundles are decomposable when $n\geq 7$.
\end{remark}


\appendix

\section{Boundedness for Fano threefolds}

In Section~\ref{s:fano3}, we worked assuming the following claim, repeated here for convenience.

\begin{assertion} \label{assn:fano3-bdd-2}
  There exists an integer $m>0$ such that for every smooth Fano threefold $X$ over an algebraically closed field of arbitrary characteristic the divisor $-mK_X$ is base-point free.
\end{assertion}

\begin{remark}
  The above claim follows from the results in \cite{ShepherdBarron}. Since we were unable to understand their proofs completely, we decided to add Assertion~\ref{assn:fano3-bdd-2} as an assumption in Theorem~\ref{thm:fano3}.
  Here are the details of this deduction. Let $X$ be a smooth Fano threefold over an algebraically closed field of characteristic $p$. If $-K_X$ is itself base-point free, there is nothing to prove. Otherwise by \cite[Theorem~3.4]{ShepherdBarron} (see also the erratum to the proof \cite[Theorem~4.1]{ShepherdBarronErrata}) either (1) $X\isom S\times \PP^1$ where $S$ is a del Pezzo surface of degree one, or (2) $X$ is isomorphic to the blowup of a smooth hypersurface of degree $6$ in the weighted projective space $\PP(1,1,1,2,3)$ along a smooth elliptic curve $C$ which is a complete intersection of degree $(1,1)$. Since varieties of types (1) and (2) form bounded families (over $\ZZ$), we can find a suitable $m$ which works for them.
\end{remark}

To avoid awkward phrasing, we will say that a smooth Fano $X$ is \emph{($F$-split)} meaning that either the base field has characteristic zero or $k$ has characterstic $p>0$ and $X$ is $F$-split. 

\begin{thm}[Boundedness for $F$-split Fano threefolds] \label{thm:boundedness}
  Assume Assertion~\ref{assn:fano3-bdd-2} holds true and the above convention is in place. Then there exist integers $m$, $N$, and $M$ such that for every smooth ($F$-split) Fano threefold $X$ over an algebraically closed field $k$ of characteristic $p \neq 2$, the linear system $|-mK_X|$ is very ample and defines an embedding $X\hookrightarrow \mathbf{P}^n$ with $n<N$ such that the coefficients of the Hilbert polynomial of the image are bounded by $M$ in absolute value. 

  In particular, there exists a scheme $H$ of finite type over $\ZZ[1/2]$ and a smooth projective morphism $\cX\to H$ of relative dimension $3$ such that for every smooth Fano threefold $X$ as above there exists a map $h\colon \Spec k\to H$ such that $X\isom h^* \cX$.
\end{thm}

The above theorem is well-known in characteristic zero \cite[Corollary~V~2.15]{kollar96}.

\begin{lemma}[Boundedness of the Hilbert polynomial] \label{lemma:bound-hilbert}
  There exists an $M$ such that for every smooth ($F$-split) Fano threefold $X$ over an algebraically closed field $k$ of characteristic $p\neq 2$, the polynomial $\chi(X, \cO_X( -t K_X))$ has coefficients bounded by $M$ in absolute value.
\end{lemma}

\begin{proof}
Because $X$ is $F$-split, it lifts to $W_2(k)$, and hence it satisfies Kodaira--Akizuki--Nakano vanishing \cite{DeligneIllusie}.  In particular for $d = \dim X$, we have the vanishing
\[
H^2(X, T_X) = H^2(X,\Omega^{d-1}_X \otimes \omega_X^{-1}) = 0.
\] 
Thus $X$ admits a formal lifting $\cX$ over $W(k)$, which is necessarily algebraizable because the ample line bundle $\omega_X^{-1}$ lifts. Let $Y = \cX \otimes {\rm Frac}(W(k))$ be the generic fiber, which is a Fano threefold over a field of characteristic zero. By flatness, we have $\chi(X, \cO_X( -t K_X)) =\chi(Y, \cO_Y( -t K_Y))$, and the latter belongs to a finite family of polynomials by boundedness in characteristic zero. 
\end{proof}

\begin{remark}
We expect the above lemma and Theorem~\ref{thm:boundedness}, to hold without the $F$-splitting assumption, at least for $p>5$. It can be deduced from the results of \cite{Das} and Assertion~\ref{assn:fano3-bdd-2} that for a fixed $p>5$, the volume $(-K_X)^3$ of a smooth Fano threefold in characteristic $p$ is bounded. It is unclear to us whether his bounds can be made independent of $p$.
\end{remark}

\begin{lemma}[Big Matsusaka for Fano threefolds] \label{lemma:big-matsusaka}
  Assume Assertion~\ref{assn:fano3-bdd-2} holds true. Then there exists an $m$ such that for every smooth Fano threefold $X$ the divisor $-mK_X$ is very ample.
\end{lemma}

\begin{proof}
By \cite{Keeler}, it is enough to show that $-mK_X$ is base-point free for a bounded $m$, which follows from Assertion~\ref{assn:fano3-bdd-2}.
\end{proof}

\begin{proof}[Proof of Theorem~\ref{thm:boundedness}]
Let $m$ be as in Lemma~\ref{lemma:big-matsusaka}, $M$ as in Lemma~\ref{lemma:bound-hilbert}, and let $N = 4Mm^3$. If $X$ is a smooth $F$-split Fano threefold over an algebraically closed field $k$ of characteristic $p>0$, then $-mK_X$ embeds $X$ into $\PP_k^n$ with $n = \dim H^0(X, \cO_X(-mK_X))$, which equals $\chi(X, \cO_X(-mK_X))$ because $H^i(X, \cO(-mK_X)) = 0$ for $i>0$ (ample line bundles on an $F$-split variety have no higher cohomology). By the bound on the coefficients of $\chi(X, \cO_X(-tK_X))$, we have $n<M+Mm+Mm^2 + Mm^3 < N$. 

We conclude that the image of $X$ in $\PP^n_k$ defines a point on the Hilbert scheme ${\rm Hilb}_\chi(\PP^n)$ where $n$ and $\chi$ both belong to a finite list of pairs $(\chi_1, n_1), \ldots, (\chi_r, n_r)$. We can now take $H$ to be the open subscheme of $\coprod {\rm Hilb}_{\chi_i}(\PP^{n_i})$ parametrizing smooth Fano varieties and $\cX$ to be (the base change of) the universal family.
\end{proof}

\bibliographystyle{amsalpha} 
\bibliography{bib.bib}

\end{document}